\documentclass[letterpaper,12pt]{article}
\usepackage{amsmath, amsthm, amssymb, graphicx,mathrsfs, polynom, setspace}
\input xy
\xyoption{all}
\theoremstyle{definition}
\newtheorem{thm}{Theorem}[section]
\newtheorem{cor}[thm]{Corollary}
\newtheorem{lem}[thm]{Lemma}
\newtheorem{pro}[thm]{Proposition}

\newtheorem{nota}[thm]{Notation}

\newtheorem{dfc}[thm]{Definition-Construction}

\newtheorem{df}[thm]{Definition}
\newtheorem{rem}[thm]{Remark}
\newtheorem{exa}[thm]{Example}

\newcommand{\mO}{\mathcal{O}}

\newcommand{\mN}{\mathcal{N}}
\newcommand{\mM}{\mathcal{M}}
\newcommand{\mI}{\mathcal{I}}

\newcommand{\mEnd}{\mathcal{E}nd}
\newcommand{\mHom}{\mathcal{H}om}

\newcommand{\bD}{\mathbb{D}}

\newcommand{\floor}[1]{\lfloor #1 \rfloor}
\newcommand{\ceiling}[1]{\lceil  #1 \rceil}

\DeclareSymbolFont{AMSb}{U}{msb}{m}{n}
\DeclareMathSymbol{\boldk}{\mathord} {AMSb}{"7C}

\begin{document}

\title{The Lemma on $b$-functions in Positive Characteristic}
\author{Theodore J. Stadnik, Jr.\footnote{Author partially supported by DMS-0354321}\\ Northwestern University and UC Berkeley}
\date{}
\maketitle

\begin{abstract}
    \centering
    \begin{minipage}
{0.6\textwidth} \

Let $X$ be an $F$-finite smooth scheme of essentially finite type over a perfect field.  This article proves the existence of $b$-functions for locally finitely generated unit $F$-modules when equipped with their induced $\bD_X$-module structure.  It is shown that the $b$-function has rational roots and is determined locally in the \'etale topology.

    \end{minipage}
\end{abstract}

\tableofcontents

\section{Introduction}

The existence of the $b$-function is an important result in the theory of $\bD$-modules in characteristic zero.  This polynomial encodes information about singularities of functions, the nearby and vanishing cycles functors, and Hodge theory.  We begin by stating an important case which motivates this theory.\\

\textbf{Theorem.} (Bernstein) If $f: \mathbb{C}^n \rightarrow \mathbb{C}$ is a meromorphic function then there exists a non-zero polynomial $b(s) \in \mathbb{C}[s]$ and a polynomial differential operator $P(s) \in \bD_{\mathbb{C}^n}[s]$ such that
$$b(s)f^s=P(s)f^{s+1}.$$
The unique minimal degree monic polynomial with this property is called the Bernstein-Sato polynomial of $f$.  If $\Gamma_f : \mathbb{C}^n \rightarrow \mathbb{C}^{n+1}$ denotes the graph of $f$, this theorem can be rephrased in terms of studying the action of the Euler vector field on the $\bD$-module ${\Gamma_f}_+\mathcal{O}_X$.\\  

The purpose of this article is to define and explore a theory of $b$-functions for $\bD_X$-modules when $X$ is a smooth variety over a field of positive characteristic and $\bD_X$ is Grothendieck's ring of differential operators. This paper generalizes the Bernstein-Sato polynomial discovered by  M. Musta{\c{t}}{\u{a}} for ${\Gamma_f}_+\mathcal{O}_X$  \cite{Mus}.  Before explaining the main results of this paper further, it will be necessary to provide a brief overview of the theory of $b$-functions over $\mathbb{C}$ and the work of M. Musta{\c{t}}{\u{a}} in positive characteristic.\\ 

Let $X$ be a smooth complex algebraic variety and $Z \subset X$ a smooth hypersurface defined globally by the sheaf of ideals $\mI_Z$. We let $\bD_X$ denote ring of differential operators.  The ideal sheaf $\mI_Z$ induces a natural filtration on the ring $\bD_X$ defined by
 $$V^{i}\bD_X = \{P \in \bD_X | P \mI_Z^{j} \subset \mI_Z^{j+i} \forall j \}.$$  This filtration is called the $V$-filtration on $\bD_X$.  More concretely, if $X = Spec(\mathbb{C}[x_1,...,x_n,t])$ and $Z=\{t=0\}$, then the $V$-filtration on $\bD_X = \mathbb{C}\langle x_1,...,x_n,t,\partial_1,...,\partial_n,\partial_t \rangle$ is given by placing $x_1,..,x_n,\partial_1,...,\partial_n$ in degree $0$, $t$ in degree $1$ and $\partial_t$ in degree $-1$.  The $0^{th}$ component of the associated graded module is $\bD_Z[-\partial_t t]$.  For the benefit of the reader, the remainder of this introduction will be restricted to this case.\\

Given a finitely generated $\bD_X$-module $\mM$ and an $\mO_X$-coherent submodule $M \subset \mM$ generating $\mM$ as a $\bD$-module, we have the $\bD_Z[-\partial_tt]$-module $N_{M}^j=V^j\bD_XM/V^{j+1}\bD_XM$.  \\

\textbf{Definition.} A $b$-function (along $Z$) for a subset $M \subset \mM$ is a (nonzero) polynomial $b(s) \in \mathbb{C}[s]$ such that $b(-\partial_tt+j)$ annihilates $N_{M}^j$ for all $j \in \mathbb{Z}$.  The minimal degree monic polynomial with this property is denoted by $b_M(s)$ and is called the $b$-function.  \\

\textbf{Definition.} The $b$-function (along $Z$) for a $\bD$-module $\mM$ is a polynomial $b(s) \in \mathbb{C}[s]$ with the following properties.
\begin{enumerate}
\item If $\lambda$ is a root of $b(s)$ then $Re(\lambda) \in (0,1]$.
\item For any (local) finite generating set $M \subset \mM$ such that $b_M(s)$ exists, if the real parts of all the roots of $b_M(s)$ are in $(0,1]$ then $b_M(s) | b(s)$.
\item Locally at every point there is a finite generating set $M \subset \mM$ such that $b_M(s)$ exists.
\item $b(s)$ is the minimal monic polynomial with the above property.

\end{enumerate}

We denote this polynomial by $b_{\mM}(s)$ when it exists and is nonzero.  It is known that for all regular holonomic $\bD$-modules $\mM$ the polynomial $b_{\mM}(s)$ exists \cite{Mal} \cite{Sab}.  Furthermore, if $b_{\mM}(s)$ exists then locally $b_M(s)$ exists for all (local) finite subsets $M \subset \mM$.  Using that $(\bD_X,V^{\cdot}\bD_X)$ is a Noetherian filtered ring, one can show that $b_{\mM}9s)$ exists if and only if (locally) for every (local) generator $M \subset \mM$ there is a nonzero polynomial $b(s)$ such that $b(-t\partial_t)$ annihilates $V^{0}\bD_XM/V^1\bD_XM$.  The theorem presented in this article will closely resemble this statement.\\

We now relate the Bernstein-Sato polynomial and the theory of $b$-functions.  Given a regular function $f: \mathbb{A}^n_{\mathbb{C}} \rightarrow \mathbb{C}$,  consider the $\bD_{\mathbb{A}^{n+1}_{\mathbb{C}}}$-module ${\Gamma_{f}}_+\mO_{\mathbb{A}^n_{\mathbb{C}}}$ where $\Gamma_f$ is the inclusion of the graph.  There is an isomorphism of $\bD_{\mathbb{A}^{n+1}_{\mathbb{C}}}$-modules

$${\Gamma_{f}}_+\mO_{\mathbb{A}^n_{\mathbb{C}}} \cong \mathbb{C}[x_1,...,x_n,t,\frac{1}{f-t}]/\mathbb{C}[x_1,...,x_n,t]$$
where the $\bD$-module structure on the right hand module is via the quotient rule.  Let $\delta \in {\Gamma_{f}}_+\mO_{\mathbb{A}^n_{\mathbb{C}}}$ denote the image of the class $[\frac{1}{f-t}]$ under this isomorphism.   It can be shown that $b_{\delta}(s)$ is the Bernstein-Sato polynomial associated to $f$.  Hence, a reasonable approach to defining Bernstein-Sato polynomials in positive characteristic would be via analysis of the action of the Euler operator $-\partial_t t$ on $\delta$.  This was successfully pursued by  M. Musta{\c{t}}{\u{a}} in \cite{Mus}.\\

The main obstacle in analyzing the action of the operator $-\partial_t t \in \bD_X$ in characteristic $p >0$ is that it satisfies the equation $Y^p - Y= 0$.  In order to make a more robust theory, \cite{Mus} considered not only $-\partial_t t$ but all of the higher Euler operators $\Theta_i = \partial_t^{[p^{i-1}]}t^{p^{i-1}}$.  Let $W$ be any $\mathbb{F}_p$ vector space acted on by $\Theta_1,...,\Theta_e$.  For a collection $\lambda_1,...,\lambda_{e} \in \mathbb{F}_p$ we define the multi-eigenspace 
$$W_{\lambda_1,...,\lambda_e} = \{ v \in V | \Theta_i(v) = \lambda_i v \text{ for all } 1 \leq i \leq e \}.$$
\cite{Mus} considered the multi-eigenspaces of the action of the Euler operators on
$$W_e(\delta) = \bD^{e}_{\mathbb{C}[x_1,...,x_n]}[t, \Theta_0,...,\Theta_e]\delta/\bD^e_{\mathbb{C}[x_1,...,x_n]}[t, \Theta_0,...,\Theta_e]t\delta$$
where $\bD^e_X$ is the subalgebra of $p^e$-linear endomorphisms of $\mO_X$. $W^e(\delta)$ is intended to serve as an approximation of the module $V^0\bD \delta / V^1 \bD \delta$ which that is used in characteristic zero.  A priori, each $W_e(\delta)$ could have as many as $p^{e}$ nontrivial multi-eigenspaces.  The main result of \cite{Mus} is that this is not possible and that the nontrivial multi-eigenspaces determine, and are determined by, a finite set of rational numbers encoding information about the singularities of $f$.\\

\textbf{Theorem 6.7 and 6.9 of } \cite{Mus}\textbf{.}
The (nontrivial) multi-eigenspaces are of $W_e(\delta)$ are completely determined by a finite set of rational numbers (the $F$-jumping exponents of $f \in (0,1])$.   This set of numbers is independent of $e$.  It follows that the number of multi-eigenspaces of $W_e(\delta)$ is uniformly bounded.  Moreover, there is an explicit relationship between the multi-eiqenspaces of $W_{e+1}(\delta)$ and the multi-eigenspaces of $W_e(\delta)$.\\

The main theorem of this paper generalizes this result to a special class of $\bD_X$-modules known as locally finitely generated unit $F^m$-modules.  Contextually, this should be understood as the positive characteristic analogue of the generalization of Bernstein-Sato polynomials to regular holonomic $\bD_X$-modules.\\

Let $R$ be a commutative domain of prime characteristic $p >0$.  A unit $F^m$-module over $R$ is a pair of an $R$-module $\mM$ and an isomorphism $\theta^{-1}: (F_R^m)^*(\mM) \rightarrow \mM $ where $(F_R^m)^*$ is pullback along the absolute Frobenius $m$ times.   Such a morphism $\theta^{-1}$ defines an action of $\bD_X$ on $\mM$ by inducing an action by each subalgebra $\bD^{me}$.  A unit $F^m$-module is said to be locally finitely generated if locally near every point there exists a coherent $R$-module $M$ and a morphism $u: M \rightarrow (F_R^{m})^*$ such that 
$$\mM \cong \varinjlim_{e}\left( M \rightarrow^{u} (F_R^m)^*M \rightarrow^{(F_R^m)^*u} (F_R^{2m})^*M \rightarrow .... \right)$$
and the morphism $\theta^{-1}$ matches the natural unit $F^m$ structure on the right hand side.\\

\textbf{Main Theorem \ref{eigenvaluelem}. }(Paraphrased) If $X$ is an $F$-finite smooth scheme of essentially finite type over a perfect field $\boldk$ of positive characteristic, $Z \subset X$ a smooth hypersurface defined locally by $\{t=0\}$, and $(\mM, F^m)$ a locally finitely generated unit $F^{m}$-module locally generated by $A: M \rightarrow F_X^{*n}M$ with $M$ is $\mO_X$-coherent then
\begin{enumerate}
\item The multi-eigenspaces of the higher Euler operators $\Theta_i = \partial_t^{[p^i]}t^{p^i}$ acting on
$$W_e(A) =  V^0\bD_{X}^eM/V^1\bD^e_{X}M$$ 
determines, and is completely determined by, a finite number of rational numbers in $(0,1]$ which are independent of $e$. 
\item The number of nontrivial multi-eigenspaces of $W_e$ is uniformly bounded (independently of $e$).
\item There is an explicit relationship between the non-trivial multi-eigenspaces of $W_e$ and $W_{e+1}$.
\item The list of rational numbers associated to the higher Euler operator actions can be made into a global invariant which is independent of the choice of local generator $A$.
\end{enumerate}

A secondary purpose of this paper is to further explore the definition of a $b$-function in positive characteristic.   Roughly speaking, because the actions of the Euler operators on $W_e(A)$ are completely determined by a finite set of rational numbers $r_1,...,r_j$, we want that $b_A(s) = \prod_i (s-r_i)$ is the $b$-function for the generator $A$.   In the characteristic $0$ case, it is often too difficult to find the minimal $b$-function and finding a function slightly bigger suffices.  Therefore in the positive characteristic setting, we would like to be able to explain what it means for a polynomial $b_A(s) \in \mathbb{R}[s]$ to sufficiently encode the action of the higher Euler operators on $W_e(A)$  but perhaps not be a minimal  choice.  Before we can make a sensible definition, we will first need to revisit the results of \cite{Mus} in more depth.\\

For any regular function $f \in \mO_X$, there is a decreasing family of ideals $\tau(f^{\lambda})$ indexed by $\lambda \in \mathbb{R}_+$ which are analogous to multiplier ideals called test ideals \cite{HY}.   If $f$ admits a lift to $\mathbb{Z}$ then for any fixed $\lambda$, $\tau(f^{\lambda})$ occurs as the mod $p$ reduction of the multiplier ideal for $p \gg 0$.  $\lambda \in \mathbb{R}_+$ is defined to be an $F$-jumping exponent if for all $\epsilon > 0$ the containment $\tau(f^{\lambda-\epsilon}) \supset \tau(f^{\lambda})$ is proper.   It is known that the set of $F$-jumping exponents is discrete and rational.  In particular, the set of $F$-jumping exponents in $(0,1]$ is finite \cite{BMS}.  We can now restate the main theorem \cite{Mus} more precisely.\\

\textbf{Theorem 6.7 of } \cite{Mus}\textbf{.} If $\lambda$ is an $F$-jumping exponent of $f$ contained in $(0,1]$ uniquely define integers $0 \leq j_i(\lambda) <p$ to be such that $$\ceiling{\lambda q^{e+1}}-1 = j_0(\lambda) + j_1(\lambda)p+...+j_{e}(\lambda)p^{e}$$ is the base $p$ expansion of $\ceiling{\lambda q^{e+1}}-1$.  If $b(s) = \prod(s-\lambda)$ where the product is taken over the $F$-jumping exponents of $f$ contained in $(0,1]$ then for each $0 \leq i \leq e$
$$ V^0\bD_R^{e}\delta/V^1\bD_R^{e}\delta=\bigoplus_{\shortstack{$\lambda$\\$b(\lambda) =0$}}\left(V^0\bD_R^{e}\delta/V^1\bD_R^{e}\delta\right)_{j_0(\lambda), j_1(\lambda),...,j_{e}(\lambda)}$$
where $\Theta_i = \partial_t^{[p^{i-1}]}t^{p^{i-1}}$ and $\delta \in (\Gamma_f)_+\mO_X$ the class of $\frac{1}{f-t}$.  \\

Guided by this theorem, we might assume that a reasonable definition for a $b$-function for an arbitrary module generated by $A: M \rightarrow F^*_R(M)$ is any polynomial $b(s)$ such that
$$W_e(A)=\bigoplus_{\shortstack{$\lambda$\\$b(\lambda) =0$}}W_e(A)_{j_0(\lambda), j_1(\lambda),...,j_{e}(\lambda)}$$

Unfortunately, this notion is overdetermined.  To see this, fix the polynomial $b(s)$ and choose $N > 0$ such that the tuples $(j_{e-N}(\lambda),...,j_e(\lambda))$ are distinct for all $e \geq N$.  We will have
$$W_e(A)_{j_0(\lambda), j_1(\lambda),...,j_{e}(\lambda)} = W_e(A)_{j_{e-N}(\lambda),...,j_{e}(\lambda)}$$

Using this we may try to modify our definition for a $b$-function for $A$ to be  any polynomial $b(s)$ such that there exists $N$ with
$$W_e(A)=\bigoplus_{\shortstack{$\lambda$\\$b(\lambda) =0$}}W_e(A)_{j_{e-N}(\lambda),...,j_{e}(\lambda)}$$
for all $e \geq N$.  This second definition is too weak because we cannot recover the roots of $b(s)$ from the eigenspace decomposition.  Luckily, we can prove something that is neither too weak nor too strong.\\

\textbf{Theorem \ref{eigenvaluelem}.} Let $R$ be a commutative $F$-finite ring which is smooth and of essentially finite type over a perfect field $\boldk$ and $A: M \rightarrow F_X^{*n}M$ a root morphism of $R[t]$-modules. If $M$ is $R[t]$-finite then there is a polynomial $b(s)$ and $N \in \mathbb{Z}$ such that the roots of $b(s)$ are rational numbers in $[0,1)$ and
$$W_e(A)=\bigoplus_{\shortstack{$\lambda$\\$b(\lambda) =0$}}W_e(A)_{j_{N}(\lambda),...,j_{e}(\lambda)}$$
for all $e \geq N$.  This prompts us to make the following definition of a $b$-function.\\

\textbf{Definition.} In the above setting, a $b$-function for the root generator $A: M \rightarrow (F_R^m)^*M$ is any polynomial $b(s) \in \mathbb{C}[s]$ such that there exists $N$ and a decomposition
$$W_e(A)=\bigoplus_{\shortstack{$\lambda \in [0,1)$\\$b(\lambda) =0$}}W_e(A)_{j_{N}(\lambda),...,j_{e}(\lambda)}$$
for all $e \geq N$.  It is clear that the set of all such $b$-functions forms an ideal in $\mathbb{C}[s]$ and we will call the minimal monic generator of this ideal $b_{A}$.\\

Our main theorem can now be restated very simply as the existence of a subtle global invariant.\\

\textbf{Main Theorem \ref{eigenvaluethm}.}  If $X$ is an $F$-finite smooth scheme of essentially finite type over a perfect field $\boldk$ of positive characteristic, $Z \subset X$ a smooth hypersurface, and $(\mM, F^m)$ a locally finitely generated unit $F^{m}$-module then there exists a unique monic polynomial $b(s)$ with rational roots associated to $\mM$ such that if (locally) $A: M \rightarrow F^{m*}_XM$ is a root generator of $\mM$ then $b(s)$ divides $b_{\bD^{(m)}_XA}(s)$ for some $m \geq 0$.\\

Here are some quick examples.\\

\textbf{Example from} \cite{Mus}\textbf{.} If $R$ is a commutative regular $F$-finite ring of essentially finite type over a perfect field $\boldk$ and
$$\mM = (\Gamma_f)_+R$$ with standard $F$-structure and $\delta$ as above then
 $$b_{\delta}(s) = \prod_{\lambda \in S}(s-\lambda)$$
where the indexing set $S$ is the set of $F$-jumping exponents for $f$ in $[0,1)$.\\

\textbf{Example \ref{tameexample}.} If $R$ is a commutative regular $F$-finite ring of essentially finite type over a perfect field $\boldk$, $p^m = 1 \mod l$, and 
$$\mM = \bD_{R[t]}\sqrt[l]{t}$$ with standard $F^m$-structure given by $h \mapsto h^{p^m}$ and root generator $\frac{\sqrt[l]{t}}{t}$ then
$b_{\frac{\sqrt[l]{t}}{t}}(s)=s-\frac{1}{l}$.\footnote{Under the Emerton-Kisin-Riemann-Hilbert correspondence this unit $F$-module corresponds to a local system that is tamely ramified at the origin.}    This is the same as $b_{\bold{C}[t,t^{-1}]\sqrt[l]{t}}(s)$ in characteristic zero.\\

\textbf{Example \ref{wildexample}.} If $R=\boldk[t]$ with $\boldk$ a perfect field and
$$\mM = \boldk[t,t^{-1},u]/(u^{p^m}+tu^{p^m-1}-t)$$ with $F^m$-structure given by the pushing forward the obvious $F^m$-structure to $\boldk[t,t^{-1}]$.\footnote{Under the Emerton-Kisin-Riemann-Hilbert correspondence this unit $F$-module corresponds to a local system that is wildly ramified at the origin.}  The $b$-function for the canonical generator divides
$$\prod_{0 \leq a < p^m}(s-\frac{a}{p^m}).$$

\textbf{Strategy of proof.}   The $b$-function as defined is determined \'etale locally so if we assume $X$ is quasi-compact it is enough to prove existence when $X = Spec(R[t)]$ for $R$ a regular ring of essentially finite type over $\boldk$ and $Z=\{t = 0\}$.  Furthermore, if $M_1 \twoheadrightarrow M_2$ then $b_{M_2}(s) | b_{M_1}(s)$ provided the latter exists.  This allows us to further reduce to the case of proving existence when $(\mM,F)$ is generated by a finitely generated free module.\\

When the generating morphism is $A:M \rightarrow F_X^{m*}M$ is free of finite rank $l$ over $R[t]$, we prove that the actions of the Euler operators on
$$V^{0}\bD_X^eM/V^{1}\bD_X^eM$$
are completely determined by certain submodules of $R^{\oplus l}$.  These submodules will be referred to as list test modules \footnote{The preferred terminology ``generalized test ideals" is already in use.} because they generalize test ideals (see \ref{testidealexa}).  For these list test modules we define the notion of jumping numbers.   The jumping numbers enjoy many of the formal properties that $F$-jumping exponents exhibit.  Abstractly, it is precisely these formal properties which allowed \cite{BMS} to conclude that the $F$-jumping exponents are discrete and rational.  Thus, the set of jumping numbers for list test modules is discrete and rational.  The polynomial with these jumping numbers as roots will be a nonzero element in the ideal of functions defining $b_{A}(s)$, proving that $b_A(s)$ exists and has rational roots.\\

The layout of the paper is as follows.  The second section contains an overview and unification of background material for use later in the paper.  The third section is dedicated to the development of the theory of ``list test modules'', their theory of jumping numbers, and the proof that these numbers are discrete and rational.  The fourth section relates the action of the Euler operators to list test modules, defines the notion of the $b$-function, and proves its existence as a consequence of the material in the third section.\\

\textbf{Acknowledgments.}  The authors reliance on the motivation provided by M. Musta{\c{t}}{\u{a}}'s analysis in \cite{Mus} will be clear to the reader.  The author would like to thank the referee for pointing out the reference \cite{Bli} which helped the author form the global definition of the $b$-function.  The referee also provided many comments and suggestions which have undoubtedly improved the presentation of this paper in various locations.

\begin{nota} Throughout this article the following conventions will be used.
\begin{itemize}
\item[-] $\boldk$ is a \textbf{perfect} field of characteristic $p$.
\item[-] $q $ is a power of $p$
\item[-] For a scheme $X$, $F_X$ is the (absolute) $q^{th}$-power Frobenius map on $X$ given by $f \mapsto f^q$.
\item[-] $R$ is a regular Noetherian commutative $\boldk$-algebra and the map $F_R$ is finite.
\item[-] $X$ is a scheme which occurs as the localization of an $F$-finite algebraic variety over $\boldk$.
\item[-] $\bD_X \subset \mEnd_{\boldk}(\mO_X)$ is Grothendieck's sheaf of differential operators.
\item[-] A map of $R$-modules, $T:M \rightarrow N$, is called $q^e$-linear if $T(r^{q^e}m)=r^{q^e}T(m)$ for all $r \in R$ and $m \in M$.
\end{itemize}
\end{nota}

\begin{rem} The use of $F_X$ to denote the $q^{th}$-power Frobenius is slightly nonstandard.  However, in this article it will be convenient to use a notation that allows us to increase the power of $p$ without complicating the notation.
\end{rem}

\section{Background}

This section will provide background information on the structural properties of certain $\bD_X$-modules called unit $F$-modules.  The results in this section are restatements, obvious generalizations, combinations, or immediate corollaries of results contained in \cite{AMBL}, \cite{Bli}, \cite{BMS1}, \cite{EK}, \cite{Has}, and \cite{Lyu}.\\

\subsection{$\bD_X$ in positive characteristic}

\begin{df}(The ring of differential operators)\\
The ring of differential operators (with divided powers) on $X$ inductively by $\bD_X \subset \mEnd_{\boldk}(\mO_X)$ via
$$\bD_X[0] = \{ P_f| P_f(g) = fg \hspace{10pt} \forall g \in \mO_X\}_{f \in \mO_X} \cong \mO_X$$
$$\bD_X[i] = \{P | [P,Q] = P \circ Q - Q \circ P \in \bD_X[i-1]\}$$
and
$$\bD_X = \cup_i \bD_X[i].$$
\end{df}

Let $\mEnd_{q^e}(\mO_X)$ be the subsheaf of $\mEnd_{\boldk}(\mO_X)$ of  $q^e$-linear endomorphisms of $\mO_X$.  That is,
$P \in  \mEnd_{q^e}(\mO_X) \Leftrightarrow P(f^{q^e}g) = f^{q^e}P(g)$ for all $f, g \in \mO_X$.
We denote the subring $\mEnd_{q^e}(\mO_X)$ by $\bD^e_X$.

\begin{thm}\label{Hasthm} \cite[1.2.5]{Has} $\bD_X = \cup_e \bD_X^e$
\end{thm}

\begin{dfc} If $S=R[t]$ then it is not difficult to see that $\partial_t^p =0$ because every polynomial differentiated $p$ times is either $0$ or multiplied by the coefficient $p$ (which is again $0$).  Given this fact, we can define a close substitute for this element by
$$\partial_t^{[p^e]}(rt^n) = \begin{cases} r\binom{n}{p^e}t^{n-p^e} &\text{ if } n \geq p^e\\
                                             0 &\text{ if } n < p^e\end{cases}$$
We can see that $$\partial_t^{[p^e]} \in \bD_S.$$  For a rigorous treatment of divided power structures see \cite{BO}.
\end{dfc}

Finally, it will be useful to understand the behavior of the subalgebras $\bD^e_X$ after \'etale pullback.  The next proposition is similar to \cite[2.14]{Bli}.
\begin{pro}\label{moritaetale}If $\pi: U \rightarrow X$ be an \'etale neighborhood then
\begin{enumerate}
\item $\pi^*\bD^e_X = \bD^e_U.$
\item  $\pi^*\mHom_{\bD^{e}_X}(\mO_X, -) \cong \mHom_{\bD^e_U}(\mO_U, \pi^*(-)).$
\end{enumerate}
\end{pro}
\begin{proof}
See \cite[2.4]{DMI}.
\end{proof}

\subsection{Unit $F_X$-modules}

Unit $F$-modules were investigated by G. Lyubeznik in \cite{Lyu} where many of their fundamental properties were first discovered.  Most notably, he defined the notion of a generating morphism and an $F$-finite unit $F$-module.  He also studied how $F$-finite unit $F$-modules, which carry a natural structure of $\bD_X$-modules, had very special properties as $\bD_X$-modules.  We begin by stating a slight variation on the definition of unit $F$-modules given in \cite{Lyu}.\\

\begin{df} A unit $F$-module is a pair $(\mM,F)$ of a quasi-coherent $\mO_X$-module, $\mM$, with an endomorphism $F: \mM \rightarrow \mM$ such that:
\begin{enumerate}
\item $F(fm) = f^qm$ for all $m \in \mM$ and $f \in \mO_X$.
\item The induced map $\theta^{-1}: F_X^*\mM \rightarrow \mM$ defined locally by $\theta^{-1}(f \otimes m) = fF(m)$ is an isomorphism.
\end{enumerate}
\end{df}

There is also a corresponding variation of the construction given for creating unit $F$-modules.\\

\begin{dfc}\cite[1.9]{Lyu} \label{limconst} Let $M$ be a quasi-coherent $\mO_X$-module and $A: M \rightarrow F_X^*M$ a map of $\mO_X$-modules. Iterating pull-back by the Frobenius map yields the directed system 
$$M \stackrel{A}{\longrightarrow} F_X^*M  \stackrel{F_X^*(A)}{\longrightarrow} F_X^{2*}M \stackrel{F_X^{2*}(A)}{\longrightarrow} ...  .$$

Let $\mM$ denote the direct limit of this directed system and $\mu_e: F_X^{e*}M \rightarrow \mM$ the inclusion maps.  By construction, $\mu_{e+1} \circ F_X^{e*}(A) = \mu_{e}$.
Define $F$ to be the endomorphism of $\mM$ given by $F(\mu_e(g \otimes m)) = \mu_{e+1}(g^q \otimes m)$.\\

The pair $(\mM,F)$ is \textbf{the unit $F$-module generated by} $(M,A)$.\\
$(M,A)$ is \textbf{a generating morphism for} $(\mM,F)$.\\
If $A$ is injective the pair $(M,A)$ is \textbf{a root morphism}.\\
$(\mM,F)$ is \textbf{locally F-finite or locally finitely generated} if locally it admits a generating morphism $(M,A)$ with $M$ $\mO_X$-coherent.\\
\end{dfc}

\begin{nota}For a generating morphism $A: M \rightarrow F_X^*M$ we define $A^{i}$ to be the composition
$$M \stackrel{A}{\longrightarrow} F_X^*M  \stackrel{F_X^*(A)}{\longrightarrow} F_X^{2*}M \stackrel{F_X^{2*}(A)}{\longrightarrow} ... \stackrel{F_X^{i*}(A)}{\longrightarrow} F_X^{(i+1)*}M  .$$
\end{nota}

\begin{rem} A unit $F$-module can be regarded as a left module for the sheaf of non-commutative algebras $\mO_X\{F\}$ which is defined to be the quotient of the free $\mO_X$-algebra on the letter $F$ by the two-sided sheaf of ideals generated by $\{f^qF-Ff\}_{f \in \mO_X}$.  It was shown in \cite[6.1.3]{EK} that a unit $F$-module $(\mM,F)$ is locally finitely generated over the ring $\mO_X\{F\}$ if and only if it is locally $F$-finite.
\end{rem}

It will be useful to have a systematic method for changing the unit $F$-structure and generating morphisms when $q$ is replaced by $q^{\gamma}$.\\

\begin{lem}\label{dropdownlem} If $q'=q^{\gamma}$ ($\gamma \geq 1$) is a power of $q$ and $F^{'}_X$ the ${q'}^{th}$-power Frobenius then any unit $F$-module $(\mM,F)$ naturally determines a unit $F^{\gamma}=F^{'}$-module structure on $\mM$ given by $(\mM,F^{'}=F^{\gamma})$.  If $(M,A)$ is a generating morphism for $(\mM,F)$ then $(M, A^{\gamma-1})$ is a generating morphism for $(\mM,F^{'})$.
\end{lem}

\subsection{$\bD_X$-actions on unit $F$-modules}

We will now begin the investigation into the relationship between unit $F$-modules and $\bD_X$-modules.  While this relationship is mentioned in \cite{Lyu}, there is an earlier and more general categorical approach presented in \cite{Has}.  The approach of the latter reference is used in this article to illustrate the category of unit $F$-modules as a full subcategory of ``periodic'' $\bD_X$-modules.\\

\begin{rem} By \ref{Hasthm},  $\bD_X \cong \cup_{e} \bD^e_X$ where $ \bD^e_X=\mEnd_{q^e}(\mO_X)$.  As $X$ is smooth, $\mEnd_{q^e}(\mO_X) \cong \mEnd_{\mO_X}(F_*^{e}\mO_X)$ is a trivial Azumaya algebra and via the Morita equivalence, the category of left modules of this ring is equivalent to the the category of left $\mO_X$-modules.
\end{rem}

The previous remark implies that we may view a left $\bD_X$-module, $\mM$, as a quasi-coherent $\mO_X$-module with compatible actions of the trivial Azumaya algebra $\bD^e_X$.  For each $e$ we can apply the Morita equivalence functor to obtain a left $\mO_X$-module $\mM_{e}$.  The modules will have certain compatibility properties in terms of $e$.  Hence, we expect that a left $\bD_X$-module may be viewed as a system of $\mO_X$-modules with compatibility conditions.  The following theorem from makes this process precise.\\

\begin{thm} \cite[2.2.4]{Has} The category of left $\bD_X$-modules is equivalent to the category of diagrams.\\
$$... \rightarrow \mM_{i+1} \rightarrow^{\phi_{i}} \mM_{i} \rightarrow ... \rightarrow \mM_0$$
subject to the following conditions:
\begin{enumerate}
\item Each $\mM_{i+1}$ is a quasi-coherent sheaf on $X$.
\item $\phi_i(fm) = f^q \phi_i(m)$ for all $i \geq 0, f \in \mO_X$ and $m \in \mM_i$.
\item $\phi_i$ induces an isomorphism $F_X^*\mM_{i+1} \rightarrow \mM_{i}$ defined locally by $f \otimes m \mapsto f\phi_i(m)$.
\end{enumerate}
\end{thm}

\begin{rem} While a proof will not be presented, it will be useful to discuss how from such a diagram we may construct a $\bD_X$ action on $\mM_0$.  Given $P \in \bD_X^e$ and $m \in \mM_0$, consider the given isomorphism $\xi^{-1}: F^{e*}_X\mM_{e} \rightarrow \mM_0$. $F^{e*}_X\mM_{e}$ \footnote{The convention in the labeling of $\xi^{-1}$ is used to be compatible with the notation in \cite{Lyu}.} carries a natural action of $\bD_X^e$ and we define $Pm = \xi^{-1} P \xi(m)$.  
\end{rem}

\begin{cor} A unit $F$-module, $(\mM,F)$, on $X$ has a natural structure of a $\bD_X$-module.
\end{cor}
\begin{proof}
Consider the diagram,
$$... \xrightarrow{F} \mM \xrightarrow{F} \mM \xrightarrow{F} ... \xrightarrow{F} \mM$$
and use the previous theorem.
\end{proof}

\begin{rem} If $q'$ is a power of $q$ then changing the $q$-structure on a unit $(\mM,F)$-module to a $q'$-structure as in \ref{dropdownlem} results in the same $\bD_X$-structures on the module $\mM$.
\end{rem}

\begin{pro}\label{dactcor} If $(\mM,F)$ is generated by $(M,A)$ then the action of $P \in End_{q^e}(\mO_X)$ on $\mu_0(m)$ is given by $P\mu_0(m) = \mu_e(P(A^{e-1}m))$ where the action of $P$ on $F_X^{e*}M$ is the natural action $P(f \otimes m') = P(f) \otimes m'$
\end{pro}

The next corollary will be useful in understanding the action of the Euler operators in terms of the generating morphism.\\

\begin{cor}\label{gencor} If $q=p^{\gamma}$, $X'=X\times_{\boldk} Spec(\boldk[t])$, $M$ is a unit $F_X$-module, and $\{\theta_i\}_{i=1}^{\gamma e} \in \bD^e_{X'}$ the operators $\theta_i = t^{p^{i-1}}\partial_t^{[p^{(i-1)}]}$ then
$$\bD_X^{e}[t,\theta_1,...,\theta_{\gamma e}]\mu_0(M) = \mu_{e}(\bD_X^{e}[t,\theta_1,...,\theta_{\gamma e}]A^{e-1}M)$$
and
$$\bD_X^{e}[t,\theta_1,...,\theta_{\gamma e}]\mu_0(tM) = \mu_{e}(\bD_X^{e}[t,\theta_1,...,\theta_{\gamma e}]tA^{e-1}M).$$\\

In particular, instead of studying the action of the Euler operators on 
$$ \bD_X^{e}[t,\theta_1,...,\theta_{\gamma e}]\mu_0(M)/\bD_X^{e}[t,\theta_1,...,\theta_{\gamma e}]\mu_0(tM)$$
it is enough to study their action on

$$\bD_X^{e}[t,\theta_1,...,\theta_{\gamma e}]A^{e-1}M/\bD_X^{e}[t,\theta_1,...,\theta_{\gamma e}]tA^{e-1}M.$$
\end{cor}

\subsection{$[\frac{1}{q^e}]$ powers of (non-unit) submodules}

This section will review and explore fractional powers of (non-unit) submodules of unit $F$-modules.  The study of fraction powers of (non-unit) submodules has origins in \cite{AMBL}.

\begin{df}\label{fracdf} If $(\mM,F)$ is a unit $F$-module and $\mN \subset \mM$ is any subsheaf of $\mO_X$-submodules (possibly not unit $F$) then define
$\mN^{[q^e]}$ to be the image under the restriction of the structure map $F^{e*}_X \mN \rightarrow \mM$ of $(\mM,F'=F^e)$ (see \ref{dropdownlem} with $\gamma = e$)). We  define fractional powers by defining $\mN^{[\frac{1}{q^e}]}$ to be the minimal module $\mN'$ with $\mN'^{[q^e]} \supset \mN$. 
\end{df}
\begin{rem}\label{fracrem}
Under the  Morita identification $\mN \cong \mHom_{\bD^e_X}(\mO_X,F^{e*}_X\mN)$ the sheaf $\mN^{[\frac{1}{q^e}]}$ is identified with the subsheaf  $\mHom_{\bD^e_X}(\mO_X,\bD^e_X\mN)$.
\end{rem}

The next proposition shows that these (fractional) powers are stable under \'etale pull-back.\\

\begin{pro}\label{etalelocalcor} If $(\mM,F)$ is a unit $F$-module and $\mN \subset \mM$ any subsheaf of $\mO_X$-submodules (possibly not unit $F$) and $\pi: Y \rightarrow X$ is a flat morphism then
\begin{enumerate}
\item $\pi^*(\mN^{[{q^e}]}) = (\pi^*\mN)^{[q^e]}.$
\item $\pi^*(\mN^{[\frac{1}{q^e}]}) = (\pi^*\mN)^{[\frac{1}{q^e}]}$ if $\pi$ is \'etale.
\end{enumerate}
\end{pro}
\begin{proof} By replacing $q$ by $q'=q^e$, it is enough to prove these statements when $e=1$.\\
\begin{enumerate}
\item 
The unit $F$ structure on $\pi^*(\mM)$, 
$$\xi^{-1}_Y: F_Y^*(\pi^*\mM) \rightarrow \pi^*\mM,$$
is defined (locally) by
$$f \otimes g \otimes m \mapsto fg^p \otimes F(m).$$

By definition,
$$\mN^{[q]}=Im(\xi^{-1}_X|(F_X^*\mN) \rightarrow \mM)$$ 
and 
$$(\pi^*\mN)^{[q]} = Im(\xi^{-1}_Y|(F_Y^*\pi^*\mN) \rightarrow \pi^*\mM)$$

There is a natural isomorphism $\eta: \pi^* \circ F_X^* \cong  F_Y^* \circ \pi^*$.  This isomorphism has the property that $\xi^{-1}_Y \eta(\mN)= \pi^* \xi^{-1}_X$.  $\pi^*$ is exact so applying it to the diagram
$$F_X^*\mN \stackrel{\xi^{-1}_X|}{\twoheadrightarrow} \mN^{[q]} \hookrightarrow \mM$$ induces the commutative diagram

$$\begin{xymatrix}{\pi^*F_X^*\mN  \ar@{>>}[d]^{\pi^*(\xi_X^{-1}|)} \ar[r]^{\eta(\mN)}_{\cong}& F_Y^*\pi^*\mN \ar[d]^{\xi_Y^{-1}|} \\ \pi^*(\mN^{[q]}) \ar@{^{(}->}[r]& \pi^*\mM }\end{xymatrix}.$$  

\item Use the definition of $\mN^{[\frac{1}{q}]}$ given in \ref{fracrem} and apply \ref{moritaetale}.

\end{enumerate}
\end{proof}

There is also the following generalization of \cite[2.5]{BMS}.

\begin{thm}\label{basistypethm} Let $l$ be an integer, $(R^{\oplus l}, F_R^{\oplus l})$ the natural unit $F$-module on $R$ of rank $l$, and $N \subset R^{\oplus l}$ a submodule generated by $v_1,...,v_n$.  If $F^e_{R*}R$ is a free $R$-module with $R$-basis $b_1,...,b_c$ $(c = q^{edim(R)})$, $b_{1,1},...,b_{1,l},...,b_{c,l}$ the corresponding basis for $F^e_{R*}R^{\oplus l}$ and $v_i = \sum_{j=0,j'=0}^{c,l} a_{i,j,j'}^{q^e}b_{j,j'}$ then setting $w_{i,j} = \langle a_{i,j,1},...,a_{i,j,l} \rangle$ we have
$$N^{[\frac{1}{q^e}]} = \sum_{i,j}Rw_{i,j}.$$
\end{thm}
\begin{proof}The proof follows by a direct adaptation of the proof of \cite[2.5]{BMS}.  To ease notation, we again note it is enough to prove the case $e=1$.\\
The inclusion $N^{[\frac{1}{q}]} \subset (\{w_{i,j}\})$ is obvious because $v_i = \sum_{j}b_jF_R(w_{i,j}).$\\
The reverse inclusion is a consequence of the following observation.  If $\{b_{j}^{\checkmark}\}$ is the dual basis of $F_{R*}R$ over $R$, then for any $N'$, $N'^{[q]}$ is closed under application of $b_j^{\checkmark}$ (acting along the diagonal).  In particular, if $N'^{[q]} \supset N$ then $b_j^{\checkmark}(v_i) = F_R(w_{i,j}) \in N'^{[q]}.$  As $F_R$ is faithfully flat, $w_{i,j} \in N'$.\\  

\end{proof}

\begin{cor}\label{genbounds} Let $R$ be a polynomial ring, $l$ an integer, and $(R^{\oplus l}, F_R^{\oplus l})$ the natural unit $F$-module on $R$ of rank $l$.  If $N \subset R^{\oplus l}$ can be generated by elements of degree at most $d$, where the degree of a vector is defined to be the maximum of the degree of each entry, then $N^{[\frac{1}{q^e}]}$ can be generated by elements of degree at most $\floor{\frac{d}{q^e}}$.
\end{cor}
\begin{proof}
Follows from \ref{basistypethm}.

\end{proof}

\subsection{Minimal generators}
Before stating a definition we will need a theorem due to M. Blickle which was first proved in the complete case by G. Lyubeznik.\\

\begin{thm}\cite[2.21]{Bli} For any $\mO_X$-coherent subsheaf $\mN \subset \mM$ the descending chain
$$\mN \supset \mN^{[\frac{1}{q}]} \supset .... \supset \mN^{[\frac{1}{q^e}]} \supset ...$$
stabilizes.  Denote by this stable value $\mN^{[0]}$.\\
\end{thm}

We begin with the definition of a minimal generator.
\begin{df}\cite[3.6]{Lyu} An $\mO_X$-coherent subsheaf $\mN \subset \mM$ is said to be minimal if $\mN$ generates $\mM$ and $\mN=\mN^{[0]}$.
\end{df}

\begin{rem} Less abstractly, if $\mN' \subset \mN$ both generate $\mM$ then it must be the case that $\mN'^{[q^e]} \supset \mN$ for some $e$ so $\mN' \supset \mN^{[\frac{1}{q^e}]}$.  Thus the definition above precisely states that $\mN$ is a minimal generator if none of its proper submodules also generate $\mM$.
\end{rem}

\begin{thm}\cite[2.24]{Bli} Any locally finitely generated unit $F$-module $\mM$ on $X$ admits a unique minimal coherent root morphism $A_{min}: M_{min} \rightarrow F^*_XM_{min}$.
\end{thm}

\subsection{$\bD^e_X$-submodules of unit $F$-modules}

The natural Morita equivalences give an explicit description of the submodule structure.  We put the statement in a format similar to \cite[2.2]{BMS1} for submodules of $R$ equipped with its canonical unit $F$-structure.

\begin{thm}\label{structurethm} If $(\mM,F)$ is a unit $F$-module then the $\bD^{e}_X$-submodules of $\mM$ are precisely those of the form $\mN^{[q^e]}$ for $\mN$ a (possibly not unit $F$) $\mO_X$-submodule of $\mM$.  In particular, the $\bD^e_X$-module generated by $\mN$ is ${\mN^{[\frac{1}{q^e}]}}^{[q^e]}$.
\end{thm}
\begin{proof}The $\bD^e_X$-structure of $\mM$ is given by the $\bD^e_R$-structure of $F^{e*}\mM$ after identification through the $e^{th}$ structure isomorphism.  By Morita equivalence, the $\bD^e_X$-submodules of $F^{e*}_X\mM$ are precisely those of the form $F^{e*}_X\mN$ for $\mN \subset \mM$ an $\mO_X$-submodule.  Applying the $e^{th}$ structure isomorphism to $\mN$ gives the result.  Alternatively, one can modify the proof of \cite[2.2]{BMS1} to obtain this result by direct computation.\\
\end{proof}

\subsection{Test ideals and $F$-jumping exponents.}

Generalized test ideals were introduced by \cite{HY} as the characteristic $p$ analogues of the multiplier ideals.  In what follows, we will work only with the affine hypersurface case and follow the presentation of \cite{BMS}.  We refer the interested reader to this source for a detailed analysis.
For this subsection, all constructions are considered for $R$ equipped with its natural unit $F$-structure.\\

\begin{df} For $f \in R$ and $\lambda \in \mathbb{R}_+$, notice that $(f^{\ceiling{\alpha q^{e}}})^{[\frac{1}{q^e}]} \subset (f^{\ceiling{\alpha q^{e+1}}})^{[\frac{1}{q^{e+1}}]}$  and define 
$$\tau(f^{\alpha}) = \cup_e (f^{\ceiling{\alpha q^{e}}})^{[\frac{1}{q^e}]}.$$\\
\end{df}

For $\alpha_1 \leq \alpha_2$ we have $\tau(f^{\alpha_1}) \supset \tau(f^{\alpha_2})$.   \cite[2.16]{BMS} shows that for every $\alpha$ there exists $\epsilon >0$ such that $\tau(f^{\alpha})=\tau(f^{\alpha+\epsilon})$. 

\begin{df}\cite[2.17]{BMS} $\alpha \in \mathbb{R}_+$ is said to be an $F$-jumping exponent for $f$ if $\tau(f^{\alpha}) \neq \tau(f^{\alpha-\epsilon})$ for all $\epsilon > 0$.  $0$ is a jumping exponent by convention.
\end{df}

\begin{rem} $\tau(f^{\alpha+1})=f\tau(f^{\alpha})$ so $\alpha$ is a jumping exponent if and only if $\alpha+1$ is a jumping exponent.
\end{rem}

The following important result will be generalized in this paper.

\begin{thm}\cite[1.1]{BMS1} The jumping exponents of $F$ are discrete in $\mathbb{R}$ and are rational.
\end{thm}

\begin{rem} \label{tirem} A simple compactness argument shows that if $\alpha_1 < \alpha_2$ are two $F$-jumping exponents with no other $F$-jumping exponents in $(\alpha_1,\alpha_2)$ then for any $\alpha_1 \leq \alpha < \alpha_2$ we have
\end{rem}
\section{List Test Modules}

We will now discuss a generalization of the test ideals $\tau(f^\alpha)$ which will be used in the next section to prove the existence of $b$-functions.  The most relevant result of \cite{BMS} to be generalized in this section is that $F$-jumping exponents are discrete and rational.  Unlike test ideals, which are defined by two parameters $\lambda$ and $e$ with $\lambda$ fixed and $e$ variable but tending towards $\infty$, list test modules will exhibit their most interesting behavior by fixing $e$ and allowing $\lambda$ to vary.  This slight modification of quantifiers requires the analogous definition of $F$-jumping exponents to be more analytic in nature.  Example \ref{testidealexa} will detail the relationship between list test modules and generalized test ideals $\tau(f^\alpha)$ and includes an analysis of the relationship between the jumping numbers and $F$-jumping exponents.  The main result of this section is that the jumping numbers associated to list test modules are discrete and rational.\\

List test modules will be determined by a list of matrices.  In order to assist the reader, in the first subsection a thorough investigation of list test modules is done when the matrices in the list are $1 \times 1$ and the list is of length less than $q$.  In this case, they will actually determine an ideal in $R$ and hence are called ``simple list test ideals''.  The generalized test ideals $\tau(f^{\alpha})$ are examples of simple list test ideals.  The second subsection contains the generalization of the important statements from simple list test ideals to the list test modules.  The proofs of the general case are omitted because they follow by the nearly identical arguments.\\

\subsection{Simple list test ideals}
To provide appropriate motivation for the definition of list test modules and their jumping numbers, we will first analyze a special case and its relation to test ideals and $F$-jumping exponents.

\begin{nota} We make a small modification of classical notion.  Let $\{S_e\}$ a union of finite subsets of $\mathbb{R}$ indexed by $\mathbb{N}$.  We say that $s$ is an accumulation point of $\{S_e\}$ if it is an element of $\cap_{e'} \overline{\cup_{e \geq e'} S_e}$.  That is, the accumulation points of $\{S_e\}$ are the usual accumulation points of $\cup_e S_e$ but we also allow limit points of constant sequences as long as the constant is contained in $S_e$ for infinitely many $e$.
\end{nota}

\begin{df} Let $r_0,...,r_{q-1} \in R$ be a list of length $q$.  For each $\lambda \in (0,1]$ and $e \geq 0$ define,

$$I(r_0,...,r_{q-1},\lambda,e) = (r_{i_0}r_{i_1}^q...r_{i_{e}}^{q^{e}})^{[\frac{1}{q^{e+1}}]}$$

where $\ceiling{\lambda q^{e+1}}-1 = i_0 + i_1q+...+i_{e}q^e$ is the unique base $q$ expansion of $\ceiling{\lambda q^{e+1}}-1$. \\

The $e^{th}$ \textbf{simple list test ideal} is,

$$\tau(r_0,...,r_{q-1},\lambda,e) = \sum_{\lambda' \leq \lambda} I(r_0,...,r_{q-1},\lambda',e).$$

This definition can be extended to all $\lambda \in \mathbb{R}$ by setting $$\tau(r_0,...,r_{q-1}, \lambda, e) = \tau(r_0,...,r_{q-1}, \lambda_0, e)$$ where $\lambda_0$ is the unique representative in $(0,1]$ of the class $\overline{\lambda} \in \mathbb{R}/\mathbb{Z}$.
\end{df}

\begin{rem} For fixed $\lambda$, the ideals $I(r_0,...,r_{q-1},\lambda,e)$ decrease as $e$ increases.  This implies that for fixed $\lambda$, the simple list test ideals also decrease as $e$ increases.  For fixed $e$, the simple list test ideals are constructed to increase as $\lambda$ increases.  There is also lower semi-continuity; for every $e$ and every $\lambda$ there exists $\lambda' < \lambda$ such that $\tau(r_0,...,r_{q-1},\lambda',e)=\tau(r_0,...,r_{q-1},\lambda,e)$.
\end{rem}

\begin{df} For each $e$, set 
$$S_e = \{ \lambda \in (0,1) | \tau(r_0,...,r_{q-1},\lambda',e) \neq \tau(r_0,...,r_{q-1},\lambda,e) \forall 1 \geq \lambda' > \lambda \}.$$  It is clear that the set $S_e$ is contained in the set $\{\frac{m}{q^{e+1}} | 0 < m < q^{e+1}\}$.\\

We define a \textbf{jumping number} for the list $r_0,...,r_{q-1}$ to be a \underline{non-zero} accumulation point of $\{S_e\}$.  A jumping number is a real number of the form $\lambda = \lim_{n \rightarrow \infty}\lambda_{e_n}$ for some sequence $\lambda_{e_n} \in S_{e_n}$ and $e_n$ a strictly increasing sequence of integers.  In the extended setting, we say $\lambda$ is a jumping number if and only if $\lambda$ is the integer translate of a jumping number.
\end{df}

\begin{rem} It will be shown in the proofs of \ref{polylem} and \ref{sltimainthm} that when $R$ is of essentially finite type over $\boldk$, $0$ is never an accumulation point of $S$.  As this is the only case that is of interest to this article, it is most convenient to exclude $0$ from the general definition.
\end{rem}
The following theorem will demonstrate that the generalized test ideals $\tau(f^{\alpha})$ for $f \in R$ are a special case of list test modules.  It will also discuss how to recover $F$-jumping exponents from the jumping numbers of the associated list test ideals.\\

\begin{thm}\label{testidealexa}
If $r_k = f^{q-1-k}$ then for all $\lambda \in (0,1]$
\begin{enumerate}
\item $I(r_0,...,r_{q-1}, \lambda, e) = \tau(r_0,...,r_{q-1}, \lambda, e)$
\item There exists $E$ such that for all $e \geq E$ and  for all $\alpha \in (0,1)$ there exist real numbers $\lambda_{\alpha,e} \in (0,1)$ with
$$\tau(r_0,...,r_{q-1}, \lambda_{\alpha,e}, e) = \tau(f^{\alpha}).$$
The numbers $\lambda_{\alpha,e}$ can be chosen so that $\alpha_1 \geq \alpha_2$ implies $\lambda_{\alpha_1,e} \leq \lambda_{\alpha_2,e}$.
\item The jumping numbers of the simple list test ideal in $(0,1)$ are of the form
$$\lambda = 1 - \alpha$$ where
$\alpha$ is an $F$-jumping exponent for $f$ contained in $(0,1)$.

\end{enumerate}
\begin{proof}\

\begin{enumerate}
\item By direct computation we have
$$I(r_0,...,r_{q-1}, \lambda, e) = (r_{i_0}r_{i_1}^q...r_{i_e}^{q^e})^{[\frac{1}{q^{e+1}}]}=(f^{\sum_{a}(q-1-i_a)q^a})^{[\frac{1}{q^{e+1}}]}=(f^{q^{e+1}-\ceiling{\lambda q^{e+1}}})^{[\frac{1}{q^{e+1}}]}.$$

This clearly implies that if $\lambda' \leq \lambda$ then $q^{e+1}-\ceiling{\lambda' q^{e+1}} \geq q^{e+1}- \ceiling{\lambda q^{e+1}}$ and
$$I(r_0,...,r_{q-1},\lambda',e) =(f^{q^{e+1}-\ceiling{\lambda' q^{e+1}}})^{[\frac{1}{q^{e+1}}]} \subset (f^{q^{e+1}-\ceiling{\lambda q^{e+1}}})^{[\frac{1}{q^{e+1}}]} = I(r_0,...,r_{q-1},\lambda,e).$$ 
We conclude that
$$\tau(r_0,...,r_{q-1}, \lambda, e)= \sum_{\lambda' \leq \lambda} I(r_0,...,r_{q-1},\lambda',e) =I(r_0,...,r_{q-1}, \lambda, e).$$
\item By \ref{tirem}, it is enough to prove this only when $\alpha$ is an $F$-jumping exponents in $(0,1)$.  $R$ is Noetherian so for any $F$-jumping exponent $\alpha \in (0,1)$ there exists $E_{\alpha}$ such that  
$$\tau(f^{\alpha}) = (f^{\ceiling{\alpha q^{e+1}}})^{[\frac{1}{q^{e+1}}]}$$
for all $e \geq E_{\alpha}$.  The set of $F$-jumping exponents form a discrete set in $\mathbb{R}_+$ so there are only a finite number of $F$-jumping exponents in $(0,1)$.  Therefore it is possible to fix an integer $E$ larger than all $E_{\alpha}$.\\

Fix an $F$-jumping exponent $\alpha \in (0,1)$ and $e \geq E$.  Let $\ceiling{\alpha q^{e+1}}-1 = i_0 + i_1q + ... + i_eq^e$ denote the base $q$-expansion of  $\ceiling{\alpha q^{e+1}}-1$.  Define $\lambda_{\alpha,e} = \sum_{k=0}^e (q-1-i_k)q^{k-e-1}$.  Using $1$ and that $e \geq E_{\alpha}$, we have
$$\tau(r_0,...,r_{q-1},\lambda_{\alpha,e},e) = (f^{q^{e+1}-\ceiling{\lambda q^{e+1}}})^{[\frac{1}{q^{e+1}}]}= (f^{\ceiling{\alpha q^{e+1}}})^{[\frac{1}{q^{e+1}}]}=\tau(f^{\alpha}).$$

\item Part 1: First we will show that if $\lambda \in (0,1)$ is a jumping number for the simple list test ideal than $\lambda$ is of the form $1-\alpha$ for $\alpha \in (0,1)$ an $F$-jumping exponent of $f$.  Write $\lambda=\lim_{e \rightarrow \infty}\lambda_{e_n}$ where $\lambda_{e_n}=\frac{m_n}{q^{e_n+1}} \in S_{e_n}$.  We want to show that $\alpha=1-\lambda$ is an $F$-jumping exponent for $f$.\\

Case 1: $1-\frac{m_n}{q^{e_n+1}} \geq \alpha$ for all $n \gg 0$.\\
By \cite[2.14]{BMS} for very large $n$, we have
$$\tau(f^{\alpha}) = (f^{m_n})^{[\frac{1}{q^{e_n+1}}]}.$$
Yet $\frac{m_n}{q^{e_n+1}} \in S_{e_n}$ so
$$\tau(f^{\alpha}) =(f^{m_n})^{[\frac{1}{q^{e_n+1}}]} = \tau(r_0,...,r_{q-1},1-\frac{m_n}{q^{e_n+1}},e_n) \subsetneq \tau(r_0,...,r_{q-1},1-\frac{m_n-1}{q^{e_n+1}},e_n)$$
for all $n \gg 0$.
We claim that $ \frac{m_n-1}{q^{e_n+1}} < \alpha$ for all $n \gg 0$; for if not by \cite[2.14]{BMS}, 
$$\tau(f^{\alpha}) = \tau(f^{\frac{m_{n_k}-1}{q^{e_{n_k}+1}}}) =  \tau(r_0,...,r_{q-1},1-\frac{m_{n_k}-1}{q^{e_{n_k}+1}},e_{n_k})$$ for some subsequence $n_k$ which would give a contradiction.  The sequence $\frac{m_n-1}{q^{e_n+1}}$ converges to $\alpha$ which implies that given $\alpha' < \alpha$ and taking $n$ large enough that $\tau(f^{\alpha'})=(f^{\ceiling{\alpha q^{e_n+1}}})^{[\frac{1}{q^{e_n}}]}$ and  $\frac{m_n-1}{q^{e_n+1}} > \alpha'$ we find that
$$\tau(f^{\alpha}) \subsetneq \tau(r_0,...,r_{q-1},1-\frac{m_n-1}{q^{e_n+1}},e_n) \subset \tau(f^{\alpha'}).$$  Therefore $\alpha$ is an $F$-jumping exponent.

Case 2: \footnote{We will see in the proofs of \ref{polylem} and \ref{sltimainthm} that the second case is actually null if $R$ is of essentially finite type over $\boldk$.  In this situation, the proof of the first case shows for very large $e$, elements of $S_e$ are always less than $\frac{1}{q^{e+1}}$ away from a jumping number.}$\frac{m_n}{q^{e_n+1}} < \alpha$ infinitely often.\\
We may pass to a subsequence and assume for all $n$ the inequality holds.\\
Fix $\alpha' < \alpha$ then chosen $n$ large enough that 
$$\tau(f^{\alpha})=(f^{\ceiling{\alpha q^{e_n+1}}})^{[\frac{1}{q^{e_n+1}}]},$$
$$ \tau(f^{\alpha'})=(f^{\ceiling{\alpha' q^{e_n+1}}})^{[\frac{1}{q^{e_n+1}}]},$$ and $\alpha' < \frac{m_n-1}{q^{e_n+1}}$. \\

Similar to the case 1 above, direct computation yields
$$\tau(f^{\alpha'}) \supset (f^{m_n-1})^{[\frac{1}{q^{e_n+1}}]} \supsetneq (f^{m_n})^{[\frac{1}{q^{e_n+1}}]} \supset \tau(f^{\alpha}).$$
 Therefore $\alpha$ is an $F$-jumping exponent.\\

Part 2: We will now show that every $F$-jumping exponent of $f$ occurs as $1-\lambda$ where $\lambda$ is a jumping number.\\

Suppose $\alpha$ is an $F$-jumping exponent.  We want to show that $\lambda=1-\alpha$ is a jumping number.  For each $n$, there is a number $\alpha'$ with $\alpha - \frac{1}{n} < \alpha' < \alpha$ and $\tau(f^{\alpha}) \neq \tau(f^{\alpha'})$. Inductively, we may choose $e_n \geq e_{n-1} \gg 0$ such that,
$$\tau(f^{\alpha}) = (f^{\ceiling{\alpha q^{e_n+1}}})^{[\frac{1}{q^{e_n+1}}]} \subsetneq (f^{\ceiling{\alpha' q^{e_n+1}}})^{[\frac{1}{q^{e_n+1}}]} = \tau(f^{\alpha'}).$$ In particular, there exists an integer $m_n$ with 
$$\ceiling{\alpha q^{e_n+1}} \geq m_n > \ceiling{(\alpha-\frac{1}{n}) q^{e_n+1}}$$ such that
$$\tau(r_0,...,r_{q-1}, 1 -\frac{m_n}{q^{e_n+1}},e_n)= (f^{m_n})^{[\frac{1}{q^{e_n+1}}]}  \subsetneq (f^{m_n-1})^{[\frac{1}{q^{e_n+1}}]} =\tau(r_0,...,r_{q-1}, 1 -\frac{m_n-1}{q^{e_n+1}},e_n).$$  Thus $1 - \frac{m_n}{q^e_n+1} \in S_{e_n}$ and the sequence $1- \frac{m_n}{q^{e_n+1}}$ converges to $\lambda$.  Thus $\lambda$ is a jumping number.
\end{enumerate}
\end{proof}
\end{thm}

\begin{pro}\label{shiftpro}
If $\lambda = \frac{m-1}{q^{e+1}} \in S_{e}$ and $m-1$ is not divisible by $q^{e}$ then the fractional part of $\frac{m-1}{q^{e}}$ is in $S_{e-1}$.
\end{pro}
\begin{proof}
We proceed by proving the contrapositive: If the fractional part of $\frac{m-1}{q^e}$ is not in $S_{e-1}$ then $\frac{m-1}{q^{e+1}}$ is not in $S_e$.\\

Write $m-1 = i_0+i_1q+....+i_eq^{e}$ and suppose that,
 $$\tau(r_0,...,r_{q-1}, \lambda_0, e-1)=\tau(r_0,...,r_{q-1}, \lambda_0 + q^{-e}, e-1)$$
where $\lambda_0 = \frac{i_0+i_1q+....+i_{e-1}q^{e-1}}{q^{e}}$ is the fractional part of $\frac{m-1}{q^{e}}$.\\

The goal is to show that
$$\tau(r_0,...,r_{q-1},\frac{m}{q^{e+1}},e) \subset \tau(r_0,...,r_{q-1},\frac{m-1}{q^{e+1}},e).$$
It is enough to show that
$$(r_{i_0}r_{i_1}^q...r_{i_e}^{q^e})^{[\frac{1}{q^{e+1}}]} \subset \tau(r_0,...,r_{q-1},\frac{m-1}{q^{e+1}}, e).$$

By the equality $\bD^{e}_RI = (I^{[\frac{1}{q^e}]})^{[q^e]}$ from \ref{structurethm} for any ideal $I \subset R$, 
$$\tau(r_0,...,r_{q-1},\lambda_0,e-1) = \tau(r_0,...,r_{q-1}, \lambda_0 + q^{-e}, e-1)$$
implies
$$\xymatrix{\tau(r_0,...,r_{q-1},\lambda_0+q^{-e},e-1)^{[q^e]} \ar@{}[r]|= \ar@{}[d]|{\rotatebox{90}{$\subset$}} & \tau(r_0,...,r_{q-1}, \lambda_0 , e-1)^{[q^e]} \ar@{}[d]|{\rotatebox{90}{$=$}}\\
\bD^{e}_Rr_{i_0}r_{i_1}^q....r_{i_{e-1}}^{q^{e-1}}  & \sum_{\lambda' \leq \lambda_0} \bD^{e}_Rr_{j_0}r_{j_1}^q....r_{j_{e-1}}^{q^{e-1}}.}$$

$\bD^{e}_R$ is defined as the operators which commute with $q^e$ powers so
$$\bD^{e}_Rr_{i_0}r_{i_1}^q....r_{i_{e-1}}^{q^{e-1}}r_{i_e}^{q^e} \subset \sum_{\lambda' \leq \lambda_0} \bD^{e}_Rr_{j_0}r_{j_1}^q....r_{j_{e-1}}^{q^{e-1}}r_{i_e}^{q^{e}}.$$

This induces the containment
$$\bD^{e+1}_Rr_{i_0}r_{i_1}^q....r_{i_{e-1}}^{q^{e-1}}r_{i_e}^{q^e} \subset \sum_{\lambda' \leq \lambda_0} \bD^{e+1}_Rr_{j_0}r_{j_1}^q....r_{j_{e-1}}^{q^{e-1}}r_{i_e}^{q^{e}} \subset \sum_{\lambda' \leq \lambda} \bD^{e+1}_Rr_{j_0}r_{j_1}^q....r_{j_{e-1}}^{q^{e-1}}r_{j_e}^{q^{e}}$$
which by a second application of \ref{structurethm} implies that
$${(r_{i_0}r_{i_1}^q...r_{i_e}^{q^e})^{[\frac{1}{q^{e+1}}]}}^{[q^{e+1}]} \subset \tau(r_0,...,r_{q-1},\frac{m-1}{q^{e+1}}, e)^{[q^{e+1}]}.$$

The proposition then follows from the faithful flatness of Frobenius.\\
\end{proof}
\begin{cor}\label{shiftthm}\label{shiftcor} \
\begin{enumerate}
\item If $\lambda \in (0,1]$ is a jumping number for $r_0,...,r_q$ then either the fractional part of $q\lambda$ is a jumping number or $\lambda = \frac{a}{q}$ with $0 < a \leq q$.
\item If $\lambda$ is a jumping number in the extended sense then either $q\lambda$ is a jumping number or $q\lambda$ is an integer.  In particular, if $q\lambda$ is not an integer then the fractional part of $q\lambda$ is a jumping number.
\end{enumerate}
\end{cor}
\begin{proof}\
\begin{enumerate}
\item We prove that if $\lambda$ is not of the form $\frac{a}{q}$ for $0 < a \leq q$ then the fractional part of $q\lambda$ is a jumping number.\\

Claim: There exists a sequence $\frac{m_n}{q^{e_n+1}} \rightarrow \lambda$ with
\begin{enumerate}
\item $\frac{m_n}{q^{e_n+1}} \in S_{e_n}$.
\item $e_n \rightarrow \infty$.
\item $m_n$ is not divisible by $q^{e_n}$.\\
\end{enumerate}

We prove the claim by contradiction.  Assume no such sequence exists.  As $\lambda$ is an accumulation point, we know there is a sequence $\frac{m_n}{q^{e_n+1}}$ converging to $\lambda$ with $\frac{m_n}{q^{e_n+1}} \in S_{e_n}$ and $e_n \rightarrow \infty$.  By the contradiction assumption, after dropping finitely many terms, we may assume $m_n$ is always divisible by $q^{e_n}$. Write $m_n=a_nq^{e_n}$ with $0 < a_n < q$ an integer.  Dividing both sides by $q^{e_n}$ and taking the limit shows that
$$q\lambda =\lim_{n \rightarrow \infty}a_n.$$

The integers $a_n$ therefore must eventually be the constant $a=q\lambda$.  In particular, $\lambda = \frac{a}{q}$, a contradiction. \\

Let $\frac{m_n}{q^{e_n+1}}$ be a sequence as in the claim.
By considering a subsequence we may assume that for all $n$ the integer $m_n$ is not divisible by $q^{e_n}$.  By the previous proposition the fractional parts of $\frac{m_n}{q^{e_n}}$ are in $S_{e_n-1}$ for all $n$.  The sequence $\frac{m_n}{q^{e_n}}$ converging to $q\lambda$ in $\mathbb{R}$ implies that the sequence $\overline{\frac{m_n}{q^{e_n}}}$ converges in $\mathbb{R}/\mathbb{Z}$ to $\overline{q\lambda}$.  As the fractional part of $q \lambda \neq 0$, the fractional part of $q\lambda$ is the unique $(0,1]$ representative of $\overline{q\lambda} \in \mathbb{R}/\mathbb{Z}$.  Therefore, the sequence of fractional parts of $\frac{m_n}{q^{e_n}}$ converges to the fractional part of $q\lambda$ and hence the fractional part of $q\lambda$ is an accumulation point of $S$.\\

\item We proceed by proving if $q\lambda$ is not an integer then $q\lambda$ is a jumping number.  By assumption, there exists an integer $j$ such that $\frac{m_n}{q^{e_n+1}}+j \rightarrow \lambda$ where $\frac{m_n}{q^{e_n+1}} \in S_{e_n}$.  As $q\lambda$ is not an integer, we know that $\lambda-j \neq \frac{a}{q}$ with $0 < a \leq q$.\\

By $1$, the fractional part of $q\lambda - qj$ is a jumping number.  Hence $q\lambda-qj$ is a jumping number (in the extended sense) which implies $q\lambda$ is a jumping number.\\
\end{enumerate}
\end{proof}

The next proposition will give us insight into the behavior of list test ideals under the formation of quotient rings.

\begin{nota}(Multi-index notation) If $x_1,...,x_n \in R$ are elements for any vector of integers $\vec{u} = \langle u_1,...,u_n \rangle$ then we define
$$x^{\vec{u}} = x_1^{u_1}x_2^{u_2}...x_n^{u_n}.$$  We also define the notation
$$0 \leq \vec{u} \leq \beta$$ for $\beta \in \mathbb{R}$ to indicate all integer valued vectors $\vec{u} = \langle u_1,...,u_n \rangle$ such that $0 \leq u_i \leq \beta$ for all $1 \leq i \leq n$.\\
\end{nota}

\begin{pro}\label{quotpro}
Fix a regular sequence $x_1,...,x_n \in R$ defining a maximal ideal $m \subset R$.  Let $Q:R \rightarrow S = R/(x_n)$ denote the quotient map and $y_i = Q(x_i)$ for $1 \leq i < n$.   Suppose that the following axioms are satisfied.
\begin{enumerate}
\item $F_{R*}^{(e+1)}R$ is freely generated over $R$ by $x^{\vec{u}}$ (in multi-index notation) for $0 \leq \vec{u} \leq q^{e+1}-1$.
\item $F_{S*}^{(e+1)}S$ is freely generated over $S$ by $y^{\vec{v}}$  (in multi-index notation) for $0 \leq \vec{v} \leq q^{e+1}-1$
\end{enumerate}
If $r_0,...,r_{q-1}$ is a list in $S$ and $\widetilde{r_0},...,\widetilde{r_{q-1}}$ are representatives in $R$ with $Proj_{x^{\vec{u}}}(\widetilde{r_k})=0$ for all $\vec{u}$ with non-zero $n^{th}$ component then
$$I(\widetilde{r_0},...,\widetilde{r_{q-1}}, \lambda, e) + (x_n) = Q^{-1}(I(r_0,...,r_{q-1},\lambda,e))$$
and 
$$\tau(\widetilde{r_0},...,\widetilde{r_{q-1}},\lambda, e) + (x_n) = Q^{-1}(\tau(r_0,...,r_{q-1},\lambda,e)).$$
\end{pro}
\begin{proof}
Write $\ceiling{\lambda q^{e+1}} - 1 = i_0  + i_1q + ... + i_e q^e$ and set
$$\widetilde{r} = \widetilde{r_{i_0}}...\widetilde{r_{i_e}}^{q^e}.$$

We want to show that $$Q^{-1}((r)^{[\frac{1}{q^{e+1}}]}) = (\widetilde{r})^{[\frac{1}{q^{e+1}}]}+(x).$$

Reindex the set $\{x^{\vec{v}} | \text{ The $n^{th}$ component of $\vec{v}$ is $0$ }\}$ by $\{b_i\}_{i=0}^{q^{(e+1)(n-1)}-1}$.  We have that $Q(b_i)$ is a basis for $F_{S*}S$ over $S$ and $b_ix_n^j$ for $0 \leq i < q^{(e+1)(n-1)}$ and $0 \leq j < q$ is a basis for $F_{R*}R$ over $R$.\\

Write $\widetilde{r} = \sum_{i,j}a_{ij}^{q^{e+1}}b_ix_n^j$ then, by the assumption $Proj_{x^{\vec{u}}}(\widetilde{r_k})=0$ for all $\vec{u}$ with non-zero $n^{th}$ component, $a_{ij}=0$ for all $j > 0$.  By \ref{basistypethm}

$$I(r_0,...,r_{q-1}, \lambda, e) = (\{\overline{a_{i0}}\})$$
and 
$$I(\widetilde{r_0},...,\widetilde{r_{q-1}}, \lambda,e) = (\{a_{ij}\})=(\{a_{i0}\}).$$

Thus,

$$Q^{-1}(I(r_0,...,r_{q-1},\lambda,e)) = (\{a_{i0} | 0 \leq i < q^{e(n-1)-1}\} \cup \{x_n\}) = I(\widetilde{r_0},...,\widetilde{r_{q-1}},\lambda,e) + (x_n).$$

The analogous result for simple list test ideals follows immediately.\\
\end{proof}
\begin{lem}\label{polylem}
If $R$ is a polynomial ring then the jumping numbers for $r_0,...,r_{q-1}$ in $(0,1]$ are finite and rational.
\end{lem}

\begin{proof}
Let $d$ be the maximum of the degrees of $\{r_i\}$.  It follows from \ref{genbounds} that $I(r_0,..,r_{q-1}, \lambda, e)$ can be generated by elements with degree less than or equal to $\frac{d(q^{e+1}-1)}{(q-1)q^{e+1}}$ for any $\lambda$ or $e$.  These numbers increase with $e$ and hence $I(r_0,...,r_{q-1},\lambda,e)$ can be generated by elements with degree less than or equal to $\floor{\frac{d}{q-1}}=\floor{\lim_e \frac{d(q^{e+1}-1)}{(q-1)q^{e+1}}}$.  By construction, the same is true for the ideals $\tau(r_0,...,r_{q-1},\lambda,e)$.\\

If $W$ is the vector space of polynomials of degree less than or equal to $\floor{\frac{d}{q-1}}$, then 
$$\tau(r_0,...,r_{q-1},\lambda',e) = \tau(r_0,...,r_{q-1},\lambda,e)$$
 if and only if they are equal after intersection with $W$.  If we set $W_{\lambda,e} = \tau(r_0,...,r_{q-1},\lambda,e) \cap W$, then for fixed $e$ the collection $\{W_{\lambda,e}\}$ is a sequence of vector subspaces of $W$.  Therefore the cardinalities of sets $S_e$ are uniformly bounded by the vector space length of $W$ plus one.\\ 

Consider the collection of finite sets $T_e = \{0\} \cup S_e.$  This collection of sets has the property that if $t$ is an accumulation point of $\{T_e$\} then the fractional part of $qt$ is also an accumulation point by \ref{shiftthm}.\\

 We will say a sequence $t_n$ of elements from $\cup_e T_e$ is strictly decreasing if $t_{n+1} < t_{n}$ and if each $t_n \in T_{e_n}$ then $e_{n+1} \geq e_n$.  We will now prove that the set $\{T_e\}$ has no strictly decreasing sequences of length more than $length(W)+2$ by showing that $\cup_e S_e$ has no strictly decreasing sequences of length more than $length(W)+1$.\\

Let $s_n \in S_{e_n}$ be a sequence of length $N'$ with $s_{n+1} < s_{n}$ and $e_{n+1} \geq e_n$. We have 

$$\tau(r_0,...,r_{q-1},s_{n+1},e_{n+1}) \subset \tau(r_0,...,r_{q-1},s_{n}, e_{n+1}) \subset \tau(r_0,...,r_{q-1}, s_n, e_n)$$
with the first containment strict as $s_{n+1} \in S_{e_{n+1}}$.  This gives us the strictly decreasing sequence with $N'$ terms,
$$... \subset \tau(r_0,...,r_{q-1}, s_{n+1}, e_{n+1}) \subset \tau(r_0,...,r_{q-1}, s_{n}, e_{n}) \subset ... \subset \tau(r_0,...,r_{q-1},s_1,e_1)$$
Repeating the argument as before, these containments are strict if and only if they are strict after intersection with $W$.  This implies $N' \leq length(W)+1$.\\

The sets $T_e$ satisfy the axioms of the next proposition and thus the set of its accumulation points are finite and rational. 
\end{proof}

\begin{pro}
If $T_e \subset [0,1]$ is a collection of subsets with the properties:
\begin{enumerate}
\item There exists $N$ such that $\{T_e\}$ contains no length $N$ sequences $t_n \in T_{e_n}$ with  $t_{n+1} < t_{n}$ and $e_{n+1} \geq e_{n}$.
\item If $\lambda$ is an accumulation point then so is the fractional part of $q\lambda$.
\end{enumerate}
then the set of accumulation points of $\{T_e\}$ is a finite set of rational numbers.
\end{pro}
\begin{proof}
We start by proving the finiteness statement by contradiction. Suppose there are distinct $N+1$ accumulation points labeled in decreasing order $\lambda_1 > \lambda_2 >... > \lambda_{N+1}$.  Choose $\epsilon$ small enough so that each accumulation point is at least $3 \epsilon$ away from the other accumulation points.  Choose $e_1$ and $t_1 \in S_{e_1}$ such that $|\lambda-t_1| < \epsilon$.  Inductively, choose $e_{i+1} \geq e_{i}$ and $t_{i+1} \in S_{e_{i+1}}$ with $|\lambda_{i+1}-t_{i+1}| < \epsilon$.  Notice that

$$t_i - t_{i+1}  \geq  -|t_i - \lambda_i| + (\lambda_i - \lambda_{i+1}) - |t_{i+1} - \lambda_{i+1}| \geq -\epsilon + 3\epsilon - \epsilon = \epsilon.$$

Thus, the $t_i$ form a strictly decreasing sequence of length $N+1$.  A contradiction to the assumption that our sets satisfy condition $1$.\\

To prove rationality, consider the set of fractional parts of $q^k\lambda$ as $k$ varies.  These are all accumulation points by condition $2$.  We have already shown that the number of accumulation points is finite so there must exist $n$ and $m$ such that the fractional parts of $q^n\lambda$ and $q^m\lambda$ are equal.  That is, $(q^n-q^m)\lambda$ is an integer $t$.  This implies $\lambda$ is the rational number $\frac{t}{q^n-q^m}$.
\end{proof}

Similar to the proof of the discreteness and rationality of $F$-jumping exponents in \cite{BMS}, one can deduce the general case for $R$ a regular $F$-finite ring of essentially finite type over $\boldk$, from the case of the polynomial ring.  First, we use the smoothness of $R$ and \ref{etalelocalcor} to reduce to the case of a standard \'etale neighborhood of a polynomial ring.  The result then follows from the case of a polynomial ring by using \ref{quotpro}.

\begin{thm}\label{sltimainthm}
If $R$ is smooth and of essentially finite type over $\boldk$ then the (extended) jumping numbers for the list $r_0,...,r_{q-1}$ are discrete and rational.
\end{thm}
\begin{proof}
As $R$ is smooth and of essentially finite type over $\boldk$, we may cover $X = Spec(R)$ by a finite collection of Zariski-open subsets $U_i \rightarrow X$ where each $U_i$ is a standard \'etale neighborhood of an affine open subset of $dim(R)$-dimensional affine space over $\boldk$.
For any map $Z \rightarrow X$ define
$$S_e(Z) = \{ \lambda \in (0,1] | \tau(r_0,...,r_{q-1},\lambda',e)|_Z \neq \tau(r_0,...,r_{q-1},\lambda,e)|_Z \forall \lambda' > \lambda \}.$$

By \ref{etalelocalcor}, it is obvious that $S_e(id:X \rightarrow X) = \cup_iS_e(U_i)$.  Therefore if $\lambda$ is an accumulation point of $\{S_e(id:X \rightarrow X)\}$ then, because the set $\{U_i\}$ is finite, $\lambda$ is an accumulation point of $\{S_e(U_i)\}$ for some $i$.  If $S$ has an infinite number of accumulation points, then for some $i$, $\{S_e(U_i)\}$ must have an infinite number of accumulation points.  Therefore it is enough to prove the lemma for each $U_i$. That is, we may assume $X$ is a standard \'etale neighborhood of an affine open subspace of affine space over $\boldk$.\\

It is enough to prove the theorem when $R$ is the localization of $\boldk[x_1,...,x_n][t]/(f(t))$ with $f(t)$ an irreducible monic polynomial.  By \ref{quotpro}, it is enough to prove the result when $R$ is the localization of a polynomial ring and, thus, enough to prove the result when $R$ is a polynomial ring.  This is \ref{polylem}.\\
 \end{proof}

The next theorem will help us to obtain information about the sets $S_e$ solely from the jumping numbers of $\{S_e\}$.

\begin{thm}
If $R$ is smooth and of essentially finite type over $\boldk$ then there exists an integer $N$ such that for all $e \geq N$
$$S_e \subset \{\frac{\ceiling{\lambda q^{e+1}}- a}{q^{e+1}} | \lambda \text{ is a jumping number for } r_0,...,r_{q-1}, 0 \leq a < q^N \}.$$
\end{thm}
\begin{proof}
The arguments from \ref{polylem} imply that there are no infinite sequences $s_n \in S_{e_n}$ with $s_{n+1} < s_{n}$ and $e_n \geq e_{n+1}$. \\

Claim 1:  For any $\epsilon >0$ such that all the jumping numbers are distance at least $3\epsilon$ away from each other, there exists $E_{\epsilon}$ such that for all $e \geq E_{\epsilon}$ and all $s \in S_e$ there exists a unique jumping number $\lambda$ with $0 \leq \lambda -s < \epsilon$.  Note that such an $\epsilon$ exists because the jumping numbers in $(0,1]$ are a finite set.\\

To prove this claim, we proceed by contradiction.  Suppose that for every $E=n$ we could find $e_n > E$ such that for some $s_n \in S_{e_n}$ every jumping number is either more than $\epsilon$ away from $s_n$ or is within distance $\epsilon$ of $s_n$ but is smaller than $s_n$.  By passing to a subsequence, we may assume the sequence $s_n$ is either always in an $\epsilon$-neighborhood of a jumping number $\lambda_0 < s_n$ or is always at least $\epsilon$ away from any jumping number.  By passing to a further subsequence, we may assume $s_n$ converges in $[0,1]$.  Now in the first case, as all jumping numbers are at least distance $3\epsilon$ apart, it must be that the numbers converge to $\lambda_0$.  However, this means that $s_n$ has a strictly decreasing subsequence which can not happen by the proof of \ref{polylem}.  By invoking the proof of \ref{polylem}, it is not possible that the sequence $s_n$ converges to $0$.  Thus in the second case, $s_n$ converges to a jumping number (by definition) yet is at least $\epsilon$ away from all other jumping numbers.  This is also a contradiction and concludes the proof of the first claim.\\

We claim we can make this statement even stronger.\\

Choose $N' \geq 2$ such that any jumping number $\lambda$ with $q\lambda$ not an integer is at least $q^{-N'-1}$ away from any number of the form $\frac{a}{q}$ with $0 \leq a \leq q$.\\

Claim 2: For every $\epsilon \in (0,q^{-N'})$ with all jumping numbers at least distance $3q\epsilon$ apart from each other, $0$, and $1$ (if it is not already a jumping number) then there exists $N_{\epsilon}$ such that for all $e \geq N_{\epsilon}$ and $s \in S_e$ there exists a unique jumping number $\lambda$ with $0 \leq \lambda - s < \epsilon q^{-e+N_{\epsilon}}$.\\

Fix $\epsilon$ as in the hypothesis of claim 2 and choose $E_{\epsilon}$ as in claim 1.   Set $N_{\epsilon} = max\{E_{\epsilon}, N'\}$ and proceed by induction on $e \geq N_{\epsilon}$.\\
  
Base case:  If $e = N_{\epsilon}$ then we are done because $e = N_{\epsilon} \geq E_{\epsilon}$ and $s \in S_e$ implies there exists a unique jumping number $\lambda$ with $0 \leq \lambda - s < \epsilon = \epsilon q^{-e+N_{\epsilon}}$.\\

Inductive step:  Assume for all $s \in S_{e-1}$ ($e-1 \geq N_{\epsilon} \geq E_{\epsilon})$ there exists a unique jumping number $\lambda$ such that $0 \leq \lambda -s < \epsilon q^{-e+1+N_{\epsilon}}$.\\

Fix $s \in S_e$.  By claim 1, there exists a unique jumping number $\lambda'$ with $0 \leq \lambda' -s < \epsilon$.  We break into cases.\\

Case 1:  $qs$ is an integer.\\

Write $s=\frac{a}{q}$ for some $0 < a < q^{e+1}$ then we have $0 \leq q\lambda'-a < q\epsilon$. Thus, $ 0 \leq fr(q\lambda') = fr(q\lambda'-a) \leq q\epsilon$ where $fr$ denotes the fractional part.  Hence, $q\lambda'$ cannot be a jumping number since all jumping numbers are distance at least $3q\epsilon$ away from $0$.  By \ref{shiftcor} we are left to conclude that $q\lambda'$ is an integer which is $q\epsilon$ away from $a$.  By choice of $\epsilon < q^{-N'}$ and $N' \geq 2$, $q\lambda' = a$ and $\lambda = s$.\\

Case 2: $qs$ is not an integer.\\

By inductive hypothesis there exists a unique jumping number $\lambda$ with $0 \leq \lambda - fr(qs) \leq \epsilon q^{-e+1+N_{\epsilon}}$.\\

It is enough to show that $s=\lambda'-\frac{\lambda}{q}+fr(qs)/q$ because then the result follows from direct computation.\\

If $q\lambda'$ is an integer then,
$$0 \leq fr(q\lambda'-qs) = 1-fr(qs) \leq q\epsilon$$
implies by choice of $\epsilon$ that $\lambda =1$.  
If $q\lambda'$ is not an integer then by \ref{shiftthm} and choice of $\epsilon$, $fr(q\lambda')=\lambda$.  In either situation, we recover that $q\lambda' - \lambda = a$ for some integer $a$.\\

Hence, $$-q\epsilon \leq q\lambda'-\lambda-(qs-fr(qs)) = a -qs+fr(qs) \leq q\epsilon$$ and the middle term is an integer.  Since $\epsilon < q^{-2}$, $a=qs-fr(qs)$ and we are done with the proof of the second claim.\\

To finish the proof of this theorem, choose $\epsilon$ meeting the criteria of claim 2 and choose $N \geq N_{\epsilon}$ so that $q^{N} > \ceiling{\epsilon q^{N_{\epsilon+1}}}$.  For $e \geq N$ and $s=\frac{m}{q^{e+1}} \in S_e$, we have $sq^{e+1} \leq \lambda q^{e+1} < sq^{e+1} + \epsilon q^{N_{\epsilon}+1}$.  Thus, 
\begin{eqnarray*}
m &=& \ceiling{sq^{e+1}}\\
&\leq& \ceiling{\lambda q^{e+1}}\\
&\leq& \ceiling{sq^{e+1}} + \ceiling{\epsilon q^{N_{\epsilon} + 1}}\\
&<& m + q^{N}
\end{eqnarray*}

We conclude, $\ceiling{\lambda q^{e+1}} \geq m > \ceiling{\lambda q^{e+1}}-q^N$.

\end{proof}

\subsection{List test modules}

We will now consider a generalization of simple list test ideals called list test modules.\\

\begin{df} Let $\{A_{k,n}\}_{k,n \geq 0}^{\infty,q-1} \in M_l(R)$ be a double-indexed set of matrices with only finitely many nonzero.  These matrices act on the free module $M=R^{\oplus l}$. For a matrix $A(t) \in M_{l}(R[t])$ write $A(t)^{[q]}$ for the matrix whose entries are the $q^{th}$ power of the entries of $A(t)$.  If we think of $A(t)$ as an $R(t)$-linear map $A(t): R[t]^{\oplus l} \rightarrow R[t]^{\oplus l} \cong F_{R[t]}^*(R[t]^{\oplus l})$ then $A(t)^{[q]} = Fr^*_{R[t]}(A(t))$.  To ease notation, the indexing set will be expanded to all of $\mathbb{Z} \times \mathbb{Z}$ by setting the undefined indices to be the zero matrix.  We will also denote by $A(t)^{e}$ the composition $A(t)^{[q^{e}]}...A(t)$.\\

Define polynomials valued in $M_l(R)$ inductively on $e$ by setting

$$H^1_n(\tau) = \sum_k A_{k,n}\tau^{k},$$
$$A(t) = \sum_{0\leq n < q} H^1_n(t^q)t^n,$$
and defining $H^{e}_n(\tau)$ to be the unique polynomials such that 
$$A(t)^{e-1} = \sum_{0 \leq n < q^{e}-1} H^{e}_n(t^{q^{e}})t^n.$$

The next lemma shows that there exists a unique minimal $N$ such that for all $e$ and $n$ 
$$H^{e}_n(\tau) \in M_n(R) \oplus M_n(R)\tau \oplus ... \oplus M_n(R) \tau^N.$$

For each $e \geq 0$ and $\lambda \in (0,1]$ define the list test module as

$$\tau(\{A_{k,j}\}, \lambda, e) = \sum_{\lambda' \leq \lambda}(H^{e+1}_{\ceiling{\lambda'q^{e+1}}-1}(\tau)R^{\oplus l})^{[\frac{1}{q^{e+1}}]} \subset R^{\oplus lN} \cong R^{\oplus l } \oplus R^{\oplus l}\tau \oplus ... \oplus R^{\oplus l}\tau^N.$$

\end{df}

\begin{lem}If $d = deg_t(A(t))$ then
$$deg_{\tau}H^e_n(\tau) \leq \frac{d}{q-1}$$
for all $e$ and $n$.
\end{lem}
\begin{proof} Recall the definition of the polynomials $H^e_n(\tau)$ as the coefficients in the expansion
$$A(t)^{e-1} = A(t)^{[q^{e-1}]}...A(t) = \sum_{n=0}^{q^{e}-1} H^{e}_n(t^{q^{e}})t^n.$$

The left side of this equation has $t$-degree at most $\frac{d(q^e-1)}{q-1}$.  Therefore
$$deg_{\tau}(H_n^e(\tau))q^e \leq deg_{\tau}(H_n^e(\tau))q^e + n = deg_t(H_n^e(t^{q^e})t^n) \leq deg_tA(t)^{e-1} \leq \frac{dq^e}{q-1}.$$
Dividing by $q^e$ yields the desired result.

\end{proof}

\begin{exa}\label{cyclicevexa} If $\{A_{k,n}\}$ be a list of matrices with the following properties,
\begin{enumerate}
\item $A_{k,n}=0$ for all $k > 0$
\item $A_{k,n} \in M_1(R)=R$ (i.e. $l=1$)
\end{enumerate}
and we set $r_i=A_{0,i}$ then $\tau(\{A_{k,n}\},\lambda,e) = \tau(r_0,r_1,...,r_{q-1},\lambda,e)$ for all $\lambda$ and $e$ where the latter ideal is the simple list test ideal.
\end{exa}
\begin{proof}
First let us prove the formula
$$H^e_{i_0+i_1q+...+i_{e-1} q^{e-1}}(\tau) = r_{i_0}r_{i_1}^q...r_{i_{e-1}}^{q^{e-1}}$$
by induction on $e$.

When $e =1$ the formula is clear.  Suppose the formula is true for $e$, that is assume
$$A^{e-1}= \sum_{i_0+i_1q+...+i_{e-1}q^{e-1}}r_{i_0}r_{i_1}^q...r_{i_{e-1}}^{q^{e-1}}t^{i_0+i_1q+...+i_{e-1}q^{e-1}}$$

This implies

$$A^{e} = \sum_{i_0+i_1q+...i_{e-1}q^{e-1}} \sum_{i_{e}}r_{i_0}r_{i_1}^q...r_{i_{e-1}}^{q^{e-1}}t^{i_0+i_1q+...+i_{e-1}q^{e-1}}r_{i_e}^{q^e}t^{i_{e}q^{e}},$$

which completes the proof by induction.

Plugging this formula back into the definition and taking $N=1$ in the definition of list test ideals we obtain that $\tau(\{A_{k,n}\},\lambda,e) = \tau(r_0,...,r_{q-1},\lambda,e)$.
\end{proof}

\begin{df} Analogously to the simple case, define
$$S_e = \{\lambda| \tau(\{A_{k,n}\},\lambda, e) \neq \tau(\{A_{k,n}\},\lambda', e) \forall \lambda' > \lambda\},$$
and the \textbf{set of jumping numbers of} $\{A_{k,n}\}$ to be the non-zero accumulation points of $\{S_e\}$.  This definition is extended to all real numbers by defining $\lambda$ to be a jumping number if some integer translate of $\lambda$ is a jumping number.
\end{df}

It will be shown in this section that this set of jumping numbers is again discrete and rational.  First, it will be necessary to generalize the results from the case of simple list test ideals.

\begin{pro}\label{generalshiftpro} If $\lambda$ is a jumping number for $\{A_{k,n}\}$ then either $q\lambda$ is an integer or $q\lambda$ is also a jumping number for $\{A_{k,n}\}$.  In particular, if $q\lambda$ is not an integer then the fractional part of $q\lambda$ is a jumping number for $\{A_{k,n}\}$.
\end{pro}
\begin{proof}
As before we proceed by contraposition and first notice that
if we write
$$H^{e-1}_{\beta}(\tau) = \sum_k B_{k,\beta}\tau^k$$ 
for any $0 \leq \beta < q^{e-1}$ then
\begin{eqnarray*}
A^{[q^{e-1}]}(t)...A(t) &=& \sum_n \sum_k \sum_{\beta} (H^1_n)^{[q^{e-1}]}(t^{q^e})t^{nq^{e-1}}B_{k,\beta}t^{kq^{e-1}}t^{\beta}\\
&=& \sum_j \sum_{n+k=j} \sum_{\beta}(H^1_n)^{[q^{e-1}]}(t^{q^e})B_{k,\beta}t^{\beta+jq^{e-1}}\\
&=& \sum_{0 \leq j_0 <q} \sum_{j_1} \sum_{n+k=j_0+j_1q} \sum_{\beta} (H^1_n)^{[q^{e-1}]}(t^{q^e})B_{k,\beta}t^{\beta+j_0q^{e-1}+j_1q^e}\\
&=& \sum_{0 \leq j_0 <q} \sum_{j_1} \sum_{n} \sum_{\beta} (H^1_n)^{[q^{e-1}]}(t^{q^e})B_{j_0+j_1q-n,\beta}t^{\beta+j_0q^{e-1}+j_1q^e}\\
\end{eqnarray*}
where $G^{[q^{e-1}]}(\tau)$ is the polynomial whose coefficients are the $[q^{e-1}]$ powers of the coefficients of $G$.\\

This implies
$$\text{ \hspace{-30pt} }(*) \text{ \hspace{30pt} }H^{e}_{\beta+j_0q^{e-1}}(\tau) = \sum_{j_1} \sum_{n} (H^1_n)^{[q^{e-1}]}(\tau)B_{j_0+j_1q-n,\beta}\tau^{j_1}.$$

By flatness of Frobenius and \ref{structurethm}, for any $e$ and $0 < \beta < q^{e-1}$, 
$$\tau(\{A_{k,n}\}, \frac{\beta+1}{q^{e-1}}, e-2) = \tau(\{A_{k,n}\}, \frac{\beta}{q^{e-1}}, e-2)$$
if and only if for every $v \in R^{\oplus l}$
$$H^{e-1}_{\beta}(\tau)v \in \sum_{\beta' < \beta} \bD^{e-1}_RH^{e-1}_{\beta'}(\tau) R^{\oplus l}.$$

In particular, if $\tau(\{A_{k,n}\}, \frac{\beta+1}{q^{e-1}}, e-2) = \tau(\{A_{k,n}\}, \frac{\beta}{q^{e-1}}, e-2)$ then there exists $\{P_{\beta'}\} \subset \bD^{e-1}_R$ and $v_{\beta'} \in R^{\oplus l}$ such that
$$\sum_{\beta' < \beta}P_{\beta'}B_{k,\beta'}v_{\beta'} = B_{k,\beta}v$$ for all $k$.\\

If $0 < \alpha < q^e$, we may write $\alpha = \beta + j_0q^{e-1}$ with $0 < \beta < q^{e-1}$, then
\begin{eqnarray}
H^{e}_{\alpha}(\tau)v &=& \sum_{j_1} \sum_{n=j_0+j_1q} (H^1_n)^{[q^{e-1}]}(\tau)B_{j_0+j_1q-n,\beta}\tau^{j_1}v\\
&=& \sum_{j_1} \sum_{n=j_0+j_1q} (H^1_n)^{[q^{e-1}]}(\tau)\sum_{\beta'< \beta}P_{\beta'}B_{j_0+j_1q-n,\beta'}v_{\beta'}\tau^{j_1}\\
&=& \sum_{\beta'< \beta}P_{\beta'} \sum_{j_1} \sum_{n=j_0+j_1q} (H^1_n)^{[q^{e-1}]}(\tau)B_{j_0+j_1q-n,\beta'}v_{\beta'}\tau^{j_1}\\
&=& \sum_{\beta'< \beta}P_{\beta'}H^1_{\beta' + j_0q^{e-1}}(\tau)
\end{eqnarray}

where the equality between $(1)$ and $(2)$ follows from $(*)$ and the equality between $(2)$ and $(3)$ follows because $P_{\beta}$ acts on $R^{\oplus l}$ through the diagonal action on each entry of $R$ and is linear with respect to $q^{e-1}$ powers by virtue of being in $\bD^{e-1}_R$.

Therefore, $$\tau(\{A_{k,n}\}, \frac{\beta+\gamma q^{e-1}+1}{q^{e}}, e-1) = \tau(\{A_{k,n}\}, \frac{\beta+\gamma q^{e-1}+1}{q^{e}}, e-1) .$$

\end{proof}

The proofs of the remaining statements follow from direct and obvious modification of the simple list test ideal case and replacing \ref{shiftpro} with the above proposition.  As such, they will not be restated.

\begin{cor} \
\begin{enumerate}
\item If $\lambda \in (0,1]$ is a jumping number for $\{A_{k,n}\}$ then either the fractional part of $q\lambda$ is a jumping number or $\lambda = \frac{a}{q}$ with $0 \leq a < q$.
\item If $\lambda$ is a jumping number in the extended sense then either $q\lambda$ is a jumping number or $q\lambda$ is an integer.  In particular, if $q\lambda$ is not an integer then the fractional part of $q\lambda$ is a jumping number.
\end{enumerate}
\end{cor}

\begin{rem} The equation $(*)$ in \ref{generalshiftpro} implies that for fixed $\lambda$ the ideals $\tau(\{A_{k,n}\},\lambda,e)$ decrease as $e$ increases.
\end{rem}

\begin{pro}
Fix a regular sequence $x_1,...,x_n \in R$ defining a maximal ideal $m \subset R$.  Let $Q:R \rightarrow S = R/(x_n)$ denote the quotient map and $y_i = Q(x_i)$ for $1 \leq i < n$.   Suppose that the following axioms are satisfied.
\begin{enumerate}
\item $F_{R*}^{(e+1)}R$ is freely generated over $R$ by $x^{\vec{u}}$ (in multi-index notation) for $0 \leq \vec{u} \leq q^{e+1}-1$.
\item $F_{S*}^{(e+1)}S$ is freely generated over $S$ by $y^{\vec{v}}$  (in multi-index notation) for $0 \leq \vec{v} \leq q^{e+1}-1$
\end{enumerate}
If $\{A_{k,n}\}$ is a list in $M_l(S)$ and $\{\widetilde{A_{k,n}}\}$ are representatives in $R$ with $Proj_{x^{\vec{u}}}(a)=0$ for all $\vec{u}$ with non-zero $n^{th}$ component and for any entry $a$ of $\widetilde{A_{k,n}}$ then
$$\tau(\{\widetilde{A_{k,n}}\}, \lambda, e) + (x_n) = Q^{-1}(\{A_{k,n}\},\lambda,e)).$$
\end{pro}

\begin{lem}
If $R$ is a polynomial ring then the jumping numbers for $\{A_{k,n}\}$ in $(0,1]$ are finite and rational.
\end{lem}

\begin{thm} If $R$ is smooth and of essentially finite type over $\boldk$ then the set of jumping numbers of $\{A_{k,n}\}$ are discrete and rational.
\end{thm}

\begin{thm}\label{Sethm}
If $R$ is smooth and of essentially finite type over $\boldk$ then there exists an integer $N$ such that for all $e \geq N$ 
$$S_e \subset \{\frac{\ceiling{\lambda q^{e+1}} - a}{q^{e+1}} | \lambda \text{ is a jumping number for } \{A_{k,n}\}, 0 \leq a < q^N \}.$$
\end{thm}

\section{Actions of Euler Operators}

Throughout this section we fix $q$ and $\gamma$ such that $q=p^{\gamma}$.

\subsection{Analysis of Euler operators}

Now we set $S = R[t]$ and use the constructions of section $2$ but for $S$ in place of $R$.\\

We will now study the actions of the higher Euler operators $\Theta_i = \partial_t^{[p^{i-1}]}t^{p^{i-1}} \in End_{p^{i}}(S) = D_S^{i}$.  We will also study the actions of the operators $\theta_i = t^{p^{i-1}}\partial_t^{[p^{i-1}]}$.  By Lucas' theorem, for any $n$ the identity $\binom{n+p^{i-1}}{p^{i-1}} = \binom{n}{p^{i-1}}$ holds in $S$.  It follows that $\Theta_i - \theta_i =1$.  The operators $\theta_i$ and $\Theta_i$ satisfy the Artin-Schreier equation $x^p-x$ and so their eigenvalues are in $\mathbb{F}_p$.  

\begin{df}Let $V$ be a $\mathbb{F}_p$-vector space equipped with commuting actions of $\{\theta_i\}_{i=1}^{\gamma e}$ (recall $q = p^{\gamma}$).  Define $0 \neq v \in V$ as an eigenvector of eigenvalue $0 \leq n < q^{e}$ if when $n= \sum_{l=0}^{\gamma e -1} i_{l}p^l$ is the base $p$ expansion then $v$ is an eigenvector of eigenvalue $i_{l-1}$ for $\theta_l$.\\
\end{df}

\begin{rem} The relation $\Theta_i - \theta_i =1$ implies that any $\mathbb{F}_p$ vector space with commuting actions of $\{\theta_i\}$ also has natural commuting actions of $\{\Theta_i\}$.  $v \in V$ is an eigenvector of eigenvalue $n = \sum_{l=0}^{\gamma e -1} i_{l}p^{l}$ for $\{\theta_i\}$ if and only if $\Theta_i v = 1+i_{l-1} v$.  This observation motivates the following  definition.
\end{rem}

\begin{df}
Let $V$ be a $\mathbb{F}_p$-vector space equipped with commuting actions of $\{\Theta_i\}_{i=1}^{\gamma e}$.  Define $0\neq v \in V$ as an eigenvector of eigenvalue $0 \leq n < q^{e}$ if when $n = \sum_{l=0}^{\gamma e-1}(p-1-j_{l})p^l$ is the base $p$ expansion then $v$ is an eigenvector of eigenvalue $-j_{l-1}$ for $\Theta_l$ for all $1 \leq l \leq \gamma e$.  Note that this definition is compatible with the definition of the eigenspaces for $\{\Theta_i\}$ given in \cite{Mus}.
\end{df}

\subsection{The $b$-function of a generating morphism}
We now fix a smooth closed codimension $1$ subscheme $Z \subset X$ defined locally by the coordinate $t=0$. 
Set $$V^{0}\bD^e_X = \{ P \in \bD^e_X | P\mI_Z^j \subset \mI_Z^j \text{ for all } j \geq 1\}$$
and $$V^{1}\bD^e_X = \{ P \in \bD^e_X | P \mI_Z^{j-1} \subset \mI_Z^{j}  \text{ for all } j \geq 1\}.$$

If $X = Spec(R[t])$ and $Z=\{t = 0\}$ then
$$V^{0}\bD^e_X = \{ P \in \bD^e_X | P\mI_Z^j \subset \mI_Z^j \text{ for all } j \geq 1\} = \bD_R^e[t,\theta_1,...,\theta_{\gamma e}]$$
and $$V^{1}\bD^e_X = \{ P \in \bD^e_X | P \mI_Z^{j-1} \subset \mI_Z^{j}  \text{ for all } j \geq 1\} = \bD_R^e[t,\theta_1,...,\theta_{\gamma e}]t.$$

The local operators $\theta_i$ are well-defined as elements of $V^{0}\bD_X^e/V^{1}\bD_X^e$.  In what follows, we only consider actions by such quotients so we make no distinction between the local operators $\theta_i$ and their lifts to $\bD^e_X$.

\begin{df}(The $b$-function of a generating morphism)
If $(\mM,F)$ is a unit $F_X$-module generated by $(M, A)$ with $M$ coherent over $X$, define a $b$-function for $((\mM,F),(M,A))$ to be any polynomial $b(s) \in \mathbb{C}[s]$ with roots in $(0,1]$ satisfying the following property: If $\{\lambda_i\}$ denotes the roots of $b(1-s)$ then there exists an integer $N$ such that for all $e \geq 0$,  the set
$$\{\ceiling{\lambda_iq^{e}} - a| 0 \leq a < q^N\} \cup \{0\}$$ contains all possible eigenvalues for the action of $\{\theta_i\}$ on
$$V^{0}\bD_X^eA^{e-1}M/V^{1}\bD_X^eA^{e-1}AM.$$

It is not difficult to check that all polynomials satisfying this property form an ideal in $\mathbb{C}[s]$.  If this ideal is not-zero then we say the $b$-function for $A$ exists and denote the unique minimal monic generator by $b_A(s)$.
\end{df}
\begin{rem} Clearly $b_A(s)$ depends on the generating morphism and $q$.
\end{rem}

\begin{rem} From the definition, it is clear the existence of a nonzero $b_A(s)$ is a Zariski-local property since $X$ is quasi-compact.  By \ref{moritaetale} it is an \'etale local property of the generator $A$.
\end{rem}

\begin{rem}The seemingly unmotivated assumption about the integer $N$ stems from the fact that the root $A$ could contain nilpotents.  The additional flexibility provided by $N$ ensures that the existence of a $b$-function is not affected by nilpotents, as will be seen in the next proposition.
\end{rem}

\begin{pro}\
\begin{enumerate}\label{nilpotentpro}
\item If $(M,A)$ is a root morphism and $b_A(s)$ exists then $b_{\bD^mA}(s)$ exists and $b_{\bD^mA}(s) | b_{A}(s)$ where $b_{\bD^mA}(s)$ is the $b$-function for the generating morphism $F^m(A)$ restricted to $\bD^{m}_XA^{m-1}M \subset F^* _XM$.
\item If there is a commutative diagram of generating morphisms on $X$
$$\begin{xymatrix}{M_1  \ar[r]^{A_1} & F_X^* M_1\\ M_2 \ar@{^{(}->}[u]^{\iota}\ar[r]^{A_2} & F^*_X M_2  \ar@{^{(}->}[u]^{F^*\iota}}\end{xymatrix}$$
such that $(M_2,A_2)$ is minimal and $b_{A_2}(s)$ exists then for some $m$ $b_{\bD^{m}A_1}(s)$ exists and $b_{\bD^{m}(A_1)}(s) | b_{A_2}(s)$.
\end{enumerate}
\end{pro}
Before proving the next proposition, we need  a lemma.

\begin{lem}\label{ringdecomplemma} Let $B \subset C$ is an inclusion of (possibly noncommutative, nonunitary) rings, $Q$ a $C$-module, and $T \subset Q$ a subset.  For $s_1,...,s_k \in C$  let   $B[s_1,...,s_{j-1}]$ denote the subalgebra of $C$ generated by $B$ and $s_1,...s_k.$  If the commutator $[c, s_{i}] \in B[s_1,...,s_{i-1}]$ for all $c \in B[s_1,...,s_{i-1}]$ then
$$B[s_1,...,s_k]T = \sum_{n \geq 0}\sum_{i=1}^k s_i^n BT.$$
\end{lem}

\begin{proof}(of \ref{nilpotentpro})

\begin{enumerate}
\item 
We wish to study the eigenspaces of $\{\theta_i\}_{i=1}^{\gamma e}$ acting on
$$(V^j\bD^{e}_X)\bD^m_X A^{e+m-1}M$$ for $j=0,1$.\\

Using that $V^j\bD^{e}_X\bD^m_X = V^j\bD^{e}_X[\partial_t, \partial_t^{[p]},...,\partial_t^{[p^{(\gamma m-1)}]}]$ (in local coordinates) then by \ref{ringdecomplemma}
we have
$$V^j\bD^{e}_X\bD^m_XA^{e+m-1}M = \sum_{n \geq 0} \sum_{i=1}^{\gamma m} (\partial_t^{[p^{i-1}]})^nV^j\bD^{e}_XA^{e+m-1}M.$$

The operators $\{\theta_i\}_{i=\gamma m+1}^{\gamma(e-m)}$ commute with $(\partial_t^{[p^{j-1}]})^n$ for any $j < \gamma m$.  Therefore if there exists a nontrivial eigenspace of value $\eta$ for the collection $\{\theta_i\}_{i=\gamma m +1}^{\gamma e}$ acting on the quotient
$$V^0\bD^{e}_X\bD^m_XA^{e+m-1}M/V^1\bD^{e}_X\bD^m_XA^{e+m-1}M$$
then it must have been true that the $\eta$ eigenspace of
$$V^0\bD^{e}_XA^{e+m-1}M/V^1\bD^{e}_XA^{e+m-1}M$$
is nontrivial.

Therefore it is enough to analyze the action of the operators $\{\theta_i\}_{i=\gamma m +1}^{\gamma e}$ on the latter quotient.
Using that $A$ is injective, we see the nontrivial eigenspaces of this quotient are exactly the nontrivial eigenspaces of 
$$V^0\bD^{e}_XA^{e-1}M/V^1\bD_X^{e}A^{e-1}M.$$

\item As $(M,A)$ is minimal, for some $m$ we have $\bD^{m}_XA_1^{m-1}M_2 = F^{m*}M_2 = \bD^{m}_XA_2M_2$.

\end{enumerate}
\end{proof}

The next proposition provides a way to simplify the analysis of the existence of $b$-functions by reducing to the case when $M$ is free of finite rank.\\

\begin{pro}\label{freered} If there is a commutative diagram of generating morphisms on $X$
$$\begin{xymatrix}{M_1 \ar@{>>}[d]^{\pi} \ar[r]^{A_1} & F_X^* M_1 \ar[d]^{F^*\pi}\\ M_2 \ar[r]^{A_2} & F^*_X M_2}\end{xymatrix}$$
and $b_{A_1}(s)$ exists then $b_{A_2}(s)$ exists and divides $b_{A_1}(s)$.
\end{pro}

\begin{proof}
$$F_X^{e*}(\pi)(V^0\bD^e_X A^{e-1}_1M_1) = V^0\bD^e_X A^{e-1}_2M_2$$
and 
$$F_X^{e*}(\pi)(V^1\bD^e_XA^{e-1}_1 M_1) = V^1\bD^e_XA^{e-1}_2M_2$$

This gives a $V^0\bD_X^e$-linear map
$$V^0\bD^e_X A^{e-1}_1M_1/V^1\bD^e_XA^{e-1}_21M_1 \twoheadrightarrow V^0\bD^e_XA^{e-1}_2 M_2/V^1\bD^e_XA^{e-1}_2 M_2$$
which proves the proposition.
\end{proof}

\subsubsection{The freely generated case in the affine setting}

It has already been discussed that the existence of $b$-functions is \'etale local on $X$.  In light of \ref{freered}, it is enough to consider the case when $X = Spec(R[t])$, $Z=Spec(R[t]/(t))$, and the generator $M$ is free.   We again set $S=R[t]$ throughout this subsection.

\begin{nota} If $A:  S^{\oplus l} \rightarrow S^{\oplus l}$ and consider $A$ as a matrix of $R$-valued polynomials in $t$.  We can also consider the composition $A^{e-1}=F_S^{(e-1)*}(A)...A$ as a matrix of polynomials.  It will be convenient to decompose these expressions in terms of the degrees of $t$ between $0$ and $q^e-1$.  Define $H^e_n(\tau) \in M_l(R[\tau])$ in the following manner,

$$A^{e-1}=F_S^{(e-1)*}(A)...A = \sum_{0 \leq n < q^{e}}H^e_n(t^{q^e})t^n.$$

\end{nota}

\begin{rem} Lemma \ref{eigenvaluelem} will show that this choice of notation $H^{e}_n(\tau)$ does not conflict with the notation from section $3$.
\end{rem}

\begin{lem}
$$\bD^e_R[t,\theta_1,...,\theta_{\gamma e}]A^{e-1}S^{\oplus l}= \sum_{0 \leq n < q^e}\bD^{e}_R[t] H_n^e(t^{q^e})t^nR^{\oplus l}$$
\end{lem}
\begin{proof}
To show that
$$\bD^e_R[t,\theta_1,...,\theta_{\gamma e}]A^{e-1}S^{\oplus l} \supset \sum_{0 \leq n < q^e}\bD^{e}_R[t] H_n^e(t^{q^e})t^nR^{\oplus l},$$
consider that for each $0 \leq n < q^e$ the operators $t^n\partial^{[n]}_t$ are in the ring generated by \{$\theta_i\}_{i=1}^e$ and $\mathbb{F}_p$.  These operators commute with $t^{q^e}$. For any $v \in R^{\oplus l}$ the polynomial $H_{q^e-1}^e(t^{q^e})t^{q^e-1}v$ is in $\bD^e_R[t,\theta_1,...,\theta_{\gamma e}]A^{e-1}S^{\oplus l}$ because it is equal to $t^{q^e-1}\partial_t^{[q^e-1]}$ applied to $A^{e-1}v$.  If $m$ is an integer and $\{H_{n}^e(t^{q^e})t^nv\}_{n > m}$ is in $\bD^e_R[t,\theta_1,...,\theta_{\gamma e}]A^{e-1}S^{\oplus l}$ then so is $H_m^e(t^{q^e})t^mv$ because it is a non-zero scalar multiple of the trailing term of $t^{m}\partial_t^{[m]}A^{e-1}v$ and the higher terms are in $\bD^e_R[t,\theta_1,...,\theta_{\gamma e}]A^{e-1}S^{\oplus l}$.  The well-ordering principle then proves the containment.\\

For the opposite containment, first consider that 
$$A^{e-1}S^{\oplus l} \subset \sum_{0 \leq n < q^e}\bD^{e}_R[t] H_n^e(t^{q^e})t^nR^{\oplus l}$$ by the definition of $H_n^e$.  Thus, it suffices to show that  $\sum_{0 \leq n < q^e}\bD^{e}_R[t] H_n^e(t^{q^e})t^nR^{\oplus l}$ is  a left $\bD^e_R[t,\theta_1,...,\theta_{\gamma e}]$-module.  For this, it is only required to show that it is closed under multiplication by $\theta_i$.  By linearity, it is enough to check for each $i$, $n$, $d \in \mathbb{N}$, and $v \in R^{\oplus l}$ that $\theta_i t^dH_n^e(t^{q^e})t^nv = \binom{n+d}{p^{i-1}}t^{d}H^e_n(t^{q^e})t^{n}v$ is contained in $\sum_{0 \leq n < q^e}\bD^{e}_R[t] H_n^e(t^{q^e})t^nR^{\oplus l}$.  This is clear.\\
\end{proof}

\begin{thm} Consider a unit $F_S$-module $(\mM, F)$ generated by $A: S^{\oplus l} \rightarrow S^{\oplus l}$.
For all $e \geq 1$ and for all $0 \leq m < q^e$, 
$$\bD^e_RH_0^e(\tau)R^{\oplus l}+...+\bD^e_RH^e_{m-1}(\tau)R^{\oplus l} = \bD^e_RH_0^e(\tau)R^{\oplus l}+...+\bD^e_RH_m^e(\tau)R^{\oplus l}$$
as subsets of $R[\tau]^{\oplus l}$ implies that
$$(\bD^e_R[t,\theta_1,...,\theta_e]A^{e-1}S^{\oplus n})_{m}=(\bD^e_R[t,\theta_1,...,\theta_e]tA^{e-1}S^{\oplus n})_{m}$$ where subscript $m$ denotes taking the $m$-th eigenspace for the operators $\{\theta_i\}$.
\end{thm}
\begin{proof} By the previous lemma,
$$\bD^e_R[t,\theta_1,...,\theta_e]A^{e-1}S^{\oplus l}= \sum_{0 \leq n < q^e}\bD^{e}_R[t] H_n^e(t^{q^e})t^nR^{\oplus l}$$

Let $x \in (\bD^e_R[t,\theta_1,...,\theta_{\gamma e}]A^{e-1}S^{\oplus l})_m$ and write 
$$x = \sum_{0 \leq n < q^e} \sum_i \sum_{0 \leq k < q^e} P_{k,n,i}(t^{q^e})t^kH_n^e(t^{q^e})t^nv_{n,i}$$
where $P_{k,n,i}(t^{q^e}) \in \bD_R^e[t^{q^e}]$ and $v_{n,i} \in R^{\oplus l}$.\\

We may rewrite this sum as

$$x =\sum_i \sum_{0 \leq u < 2q^e-1}\sum_{0 \leq n < q^e}P_{u-n,n,i}(t^{q^e})H_n^{e}(t^{q^e})t^uv_{n,i}.$$

By Lucas' theorem 
$$\theta_l P_{u-n,n,i}(t^{q^e})H_n^{e}(t^{q^e})t^uv_{n,i} = \binom{u}{p^{l-1}}P_{u-n,n,i}(t^{q^e})H_n^{e}(t^{q^e})t^uv_{n,i}.$$

Yet $x$ lives in the $m^{th}$ eigenspace so in this expansion of $x$, we can restrict to taking the sum only over those $u$ which are equal to $m$ modulo $q^e$ (of which there are at most two).  We get the expression
$$x = \sum_i\left(\sum_{0 \leq n < q^e} P_{m-n,n,i}(t^{q^e})H_n^e(t^{q^e})t^mv_{n,i} + \sum_{0 \leq n < q^e}P_{m+q^e-n,n,i}(t^{q^e})H_n^e(t^{q^e})t^{m+q^e}v_{n,i}\right).$$

Many of these summands are zero because $P_{k,n,i}(t^{q^e})$ are defined (by convention) to be zero when $k < 0$.  Thus,
$$x = \sum_i\left(\sum_{0 \leq n \leq m} P_{m-n,n,i}(t^{q^e})H_n^e(t^{q^e})t^mv_{n,i} + \sum_{m < n < q^e}P_{m+q^e-n,n,i}(t^{q^e})H_n^e(t^{q^e})t^{m+q^e}v_{n,i}\right).$$

For $n < m$, $H_n^{e}(t^{q^e})t^mv_{n,i} = t^{m-n}H_n^{e}(t^{q^e})t^nv_{n,i}$ is in $(\bD^e_R[t,\theta_1,...,\theta_{\gamma e}]tA^{e-1}S^{\oplus l})_{m}$ by the previous lemma. Also, $\sum_{m < n < q^e}P_{m+q^e-n,n,i}(t^{q^e})H_n^e(t^{q^e})t^{m+q^e}v_{n,i}$ is in $(\bD^e_R[t,\theta_1,...,\theta_{\gamma e}]tA^{e-1}S^{\oplus l})_{m}$.\\

Thus to show that $x \in (\bD^e_R[t,\theta_1,...,\theta_{\gamma e}]tA^{e-1}S^{\oplus l})_{m}$ it is enough to show that the $m^{th}$ term $\sum_i P_{0,m,i}(t^{q^e})H_m^e(t^{q^e})t^mv_{m,i}$ is in $(\bD^e_R[t,\theta_1,...,\theta_{\gamma e}]tA^{e-1}S^{\oplus l})_{m}$.\\

Write $P_{0,m,i}(t^{q^e}) = \sum_j P_{j,i} t^{jq^e}$ where $P_{j,i} \in \bD_R^e$.  By assumption, for each $i$ there exist $Q_{j,w,i} \in \bD^e_R$ and $v_{j,w,i} \in R^{\oplus l}$ such that $\sum_{0 \leq w < m} Q_{j,w,i}H_w^e(\tau)v_{j,w,i} = P_{j,i}H_m^e(\tau)v_{m,i}$.  Setting $\tau = t^{q^e}$, multiplying both sides by $t^{jq^e}t^m$, and summing over $i,j$ we obtain
\begin{eqnarray*}
\sum_i \sum_j \sum_{0 \leq w < m} Q_{j,w,i}t^{jq^e}t^{m-w}H_w(t^{q^e})t^wv_{j,w,i} &=& \sum_i \sum_j P_{j,i}t^{jq^e}H^e_m(t^{q^e})t^mv_{m,i}\\
 &=& \sum_i P_{0,m,i}(t^{q^e})H^e_m(t^{q^e})v_{m,i}.
\end{eqnarray*}
 
By the previous lemma, the left side of the above equation is contained in $(\bD^e_R[t,\theta_1,...,\theta_{\gamma e}]tA^{e-1}S^{\oplus l})_{m}$ (the sum is over $w < m$).

\end{proof}

\begin{cor}\label{weightcor}
If there is a non-zero eigenvector of weight $m$ for the operators $\{\theta_i\}$ in 
$$\bD^e_R[t,\theta_1,...,\theta_e]A^{e-1}S^{\oplus l}/\bD^e_R[t,\theta_1,...,\theta_e]tA^{e-1}S^{\oplus l}$$ 
then $$\bD^e_RH_0^e(\tau)R^{\oplus l}+...+\bD^e_RH^e_{m-1}(\tau)R^{\oplus l} \neq \bD^e_RH_0^e(\tau)R^{\oplus l}+...+\bD^e_RH_m^e(\tau)R^{\oplus l}.$$
\end{cor}

\subsubsection{Relationship to list test modules}

In this subsection, we continue working in the context of the previous subsection.  We will show, similar to the case of the first local cohomology module, that the (infinite) behavior of the eigenvalues of $\{\theta_i\}$ is completely controlled by the jumping numbers of list test modules when $R$ is smooth and of essentially finite type over $\boldk$.  First, we begin by relating the eigenvalues to the sets $S_e$.

\begin{lem}\label{eigenvaluelem} If $(\mM,F)$ is generated by $(M,A(t))$, $H_n^1(\tau) = \sum_kA_{k,n}\tau^k$, and $S_e$ is the set associated to $\{A_{k,n}\}$ in Section $3$, then the eigenvalues of the action of $\{\theta_i\}_{i=1}^{\gamma e}$ on
$$\bD^e_R[t,\theta_1,...,\theta_{\gamma e}]F^{(e-1)*}_S(A)...AS^{\oplus l}/\bD^e_R[t,\theta_1,...,\theta_{\gamma e}]tF^{(e-1)*}_S(A)...AS^{\oplus l}$$ 
are contained in the set $\{m| \frac{m}{q^{e}} \in S_{e-1} \} \cup \{0\}$.
\end{lem}
\begin{proof}
In the canonical basis for $S^{\oplus l}$, $F^{e*}A=A(t)^{[q^e]}$ in the notation of section $3$.  The definition of the matrix $A(t)^{e-1}$ given in Section $3$ coincides with the definition of the generator $A(t)^{e-1}$ given in Section $4$.  Hence, the corresponding definitions for $H_n^e(\tau)$ that we attached to $A^{e-1}$ yield the same polynomials $H_n^e(\tau)$ in the list test module and generating morphism cases.\\

By \ref{weightcor}, if there is a non-zero eigenvector of weight $m \neq 0$ in $$\bD^e_R[t,\theta_1,...,\theta_{\gamma e}]A^{e-1}S^{\oplus l}/\bD^e_R[t,\theta_1,...,\theta_{\gamma e}]tA^{e-1}S^{\oplus l}$$ then
$$\bD^e_RH_0^e(\tau)R^{\oplus l}+...+\bD^e_RH^e_{m-1}(\tau)R^{\oplus l} \neq \bD^e_RH_0^e(\tau)R^{\oplus l}+...+\bD^e_RH_m^e(\tau)R^{\oplus l}$$
which implies
$$\tau(\{A_{k,n}\},\frac{m}{q^{e}},e-1)^{[q^e]} \neq \tau(\{A_{k,n}\},\frac{m+1}{q^{e}},e-1)^{[q^e]}.$$
By the faithful flatness of Frobenius, the last inequality is true is if and only if $\frac{m}{q^{e}} \in S_{e-1}$.\\

\end{proof}

\subsubsection{Global $b$-functions}

The next theorem explicitly describes how the jumping numbers control the (infinite) behavior of the eigenvalues of the operators $\theta_i$ when is $X$ is any smooth $F$-finite scheme of essentially finite type over $\boldk$.

\begin{thm}\label{eigenvaluethm} Suppose $(\mM,F)$ a unit $F$-module generated on $X$ which is locally generated.  
\begin{enumerate}
\item For any coherent generator $A: M \rightarrow F^*_XM$ the polynomial $b_A(s)$ exists and has rational roots.
\item If $A_{min}$ is the minimal coherent root morphism of $\mM$ then
\begin{enumerate}
\item $b_{A_{min}}(s)$ exists and has rational roots.
\item For any root morphism $A:M \rightarrow F^*_XM$, there exists $m$ such that the polynomial $b_{\bD^{m}A}(s)$ divides $b_{A_{min}}(s)$.

\end{enumerate}
\end{enumerate}
\end{thm}
\begin{proof}\

\begin{enumerate}
\item Working \'etale locally, it will follow directly from \ref{Sethm} and \ref{eigenvaluelem}.
\item $(a)$ is clear from $1$.  $(b)$ follows from \ref{nilpotentpro}.
\end{enumerate}
\end{proof}
As the minimal root generator $A_{min}$ is a global invariant of the unit $F$-module $(\mM,F)$, the previous theorem motivates us to make the following definition

\begin{df}(The global $b$-function) If $(\mM,F)$ is a locally finitely generated unit $F$-module on $X$ then the global $b$-function is
$b_{\mM}(s) = b_{A_{min}}(s)$ where $A_{min}$ is the unique minimal coherent root morphism of $(\mM,F)$.
\end{df}

\subsection{Examples}
\begin{exa}(Free cyclic generators of low degree) If $l=1$ and $A(t): R[t] \rightarrow R[t]$ is such that each $H^1_n(\tau)$ is constant, then the roots of $b_A(s)$ are of the form $1-\lambda$ where $\lambda$ varies over the jumping numbers for the list $H^1_0,...,H^1_{q-1}$.  This follows from \ref{cyclicevexa}.  Conversely, given a list $r_0,...,r_{q-1} \in R$ we can construct an associated unit $F$-module on $R[t]$ generated by $A(t) = \sum_{0 \leq i < q} r_0 t^i: R[t] \rightarrow R[t]$ such that $b_A(s)$ is determined by the jumping numbers of the list.
\end{exa}
\begin{exa}(The free resolution of the first local cohomology module)
Let $f \in R$, $S=R[t]$, and $\Gamma_f: Spec(R) \rightarrow Spec(S)$ the graph of $f$.  The first local cohomology module $(\Gamma_f)_+\mO_X$ has coherent generator
 $$\delta: R=S/(f-t)S \rightarrow S/(f-t)S=R \text{ by } \overline{p(t)} \mapsto (f-t)^{p-1}\overline{p(t)}.$$ 
and free resolution $\tilde{d}: S \rightarrow S$ by multiplication by $(f-t)^{p-1}$.\\

Combining the previous example with \ref{testidealexa},  $b_{\tilde{\delta}}$ and $b_{\delta}(s)$ have roots contained in the set of $F$-jumping exponents of $f$ in $[0,1)$.  In \cite{Mus} it is shown that the roots of $b_{\delta}(s)$  are precisely the $F$-jumping exponents of $f$ in $[0,1)$. This is the characteristic $p$ analogue of the characteristic $0$ statement (see \cite{BS} or \cite{ELSV}) that if $\lambda$ is a jumping exponent then $b(-\lambda)=0$.
\end{exa}

\begin{exa}\label{tameexample}(Pushforward of rank one tame local systems onto $\mathbb{A}^1$)\
Let $X=\mathbb{A}^1=Spec(\boldk[t])$, $U=Spec(\boldk[t,t^{-1}])$, and $m$ an integer dividing $q-1$.  Consider the $\boldk[t]$-module $\mM=\boldk[t,t^{-1}]\sqrt[m]{t}$ and its obvious Frobenius map $F(f\sqrt[m]{t})=f^qt^{\frac{q-1}{m}}\sqrt[m]{t}$.  This gives $\mM$ the structure of a unit $F$-module on $X$.  $\mM$ is generated as a unit $F$-module by the morphism $A_j:\boldk[t] \rightarrow \boldk[t]$ defined as $A: f \mapsto ft^{j(q-1)-\frac{q-1}{m}}$ for any $j \geq 1$.  This generating morphism is the same as choosing the $\bD_X$-module generator $t^{-j}\sqrt[m]{t}$ because $\mu_0(\boldk[t]) = \boldk[t]t^{-j}\sqrt[m]{t} \subset \boldk[t,t^{-1}]\sqrt[m]{t}$.\\

 In characteristic $0$, $b_{t^{-j}\sqrt[m]{t}}(s)=(s+1-j+\frac{1}{m})$ and $b_{\mathbb{C}[t,t^{-1}]\sqrt[m]{t}}(s)=(s-\frac{1}{m})$.  We will show that $b_{A_j}(s)=(s-\frac{1}{m})$.  To ease notation, set $\beta_e = j+\frac{q^e-1}{m}$ then for all $e \gg 0$, 
$$H^e_n(\tau) = \begin{cases}  0 & \text{ if } n \neq q^e-\beta_e\\
															\tau^{j-1} & \text{ if } n = q^{e}-\beta_e\\
															\end{cases}$$
															
Therefore, either using \ref{eigenvaluelem} or the theory developed for simple list test ideals, $S_e=\{\frac{q^e-\beta_e}{q^e}\}$.  
The only  limit point of the sets $\{S_e\}$ is clearly $(1-\frac{1}{m})$ and so $b_{A_j}(s-\frac{1}{m})$.  Notice that this is independent of $j$.  
\end{exa}

\begin{exa}\label{wildexample}(Pushforward of the wildly-ramified Artin-Schreier local system onto $\mathbb{A}^1$)\
Let $X=\mathbb{A}^1=Spec(\boldk[t])$, $Z=Spec(\boldk[t]/(t))$, and $\mM=\boldk[t,t^{-1},u]/(u^q+tu^{q-1}-t)$ considered as a $\boldk[t]$-module with Frobenius morphism $F(f(t,u))=f(t,u)^q$.  This Frobenius makes $\mM$ into a unit $F$-module over $X$ which is of rank $q$ over $X\setminus Z$.  It is the pushforward of the structure sheaf $F$-crystal on the \'etale locus of an Artin-Schreier cover that is wildly ramified at the origin.  A coherent generator \footnote{This is the generator obtained by taking the obvious structure morphism on the Artin-Schreier cover of $\mathbb{P}^1$ and pushing it forward to $\mathbb{A}^1$} of this module is $M = \boldk[t]^{\oplus q}$ and
$$A(e_j) = t^{q-1}e_j - \sum_{i=0}^{j-1}\binom{j}{i} t^{q-1-j+i}e_i$$ where is the standard basis $e_0,...,e_{q-1}$ of $\boldk[t]^{\oplus q}$.\\

If we define negative indicies to be $0$ then
$$H_n^1(\tau)e_j =H_n^{1}e_j= \begin{cases} -\binom{j}{n+j+1-q}e_{n+j+1-q} & n \neq q-1 \\ e_j & n =q-1 \end{cases}$$
and
$$H_{i_0+i_1q+...+i_{e-1}q^{e-1}}^e = H_{i_{e-1}}^1...H_{i_1}^1H^1_{i_0}.$$

The only jumping numbers of the list test module are of the form $\frac{mq^{e-1}}{q^{e}}$ for $0 < m \leq q-1$.  It follows that $b_A(s)$ divides
$\prod_{0 \leq a < q} (s - \frac{a}{q})$.
\end{exa}

Department of Mathematics, Northwestern University\\
2033 Sheridan Road, Evanston, IL 60208\\
\textit{E-mail: tstadnik@math.northwestern.edu}\\
\end{document}